\documentclass[11pt]{article}

\usepackage[a4paper,margin=1.02in]{geometry}
\usepackage[T1]{fontenc}
\usepackage[utf8]{inputenc}
\usepackage{lmodern}
\usepackage{amsmath,amssymb,amsthm,mathtools}
\usepackage{microtype}
\usepackage{enumitem}
\usepackage{tabularx,booktabs}
\usepackage[hidelinks]{hyperref}
\usepackage{xcolor}

\setlist{nosep}
\numberwithin{equation}{section}

\newtheorem{theorem}{Theorem}
\newtheorem{lemma}{Lemma}
\newtheorem{proposition}{Proposition}
\newtheorem{corollary}{Corollary}

\theoremstyle{definition}
\newtheorem{definition}{Definition}
\theoremstyle{plain}

\newcommand{\Fix}{\operatorname{Fix}}
\newcommand{\Sub}{\operatorname{Sub}}
\newcommand{\G}{\mathcal G}
\newcommand{\M}{\mathcal M}

\newcommand{\supp}{\operatorname{supp}}
\newcommand{\dom}{\operatorname{dom}}
\newcommand{\ran}{\operatorname{ran}}
\newcommand{\id}{\operatorname{id}}
\newcommand{\dist}{\operatorname{dist}}
\newcommand{\vi}{\vskip 0.1in \noindent}

\newcommand{\Ball}{B}
\newcommand{\whG}{\widehat G}
\newcommand{\whX}{\widehat X}
\newcommand{\preceqG}{\preceq_G}
\newcommand{\preceqcl}{\preccurlyeq}
\newcommand{\1}{\mathbf 1}

\begin{document}
\title{\bfseries The Almost Finite--Purely Infinite Dichotomy\\
for Minimal Amenable Ample Groupoids}
\author{G\'abor Elek \and \'Ad\'am Tim\'ar}
\date{}

\hypersetup{pdftitle={The Almost Finite--Purely Infinite Dichotomy},pdfsubject={matuikesz v2: Introduction v3 and Preliminaries v15},pdfauthor={Gabor Elek and Adam Timar}}
\maketitle
\begin{abstract}
We answer, in the principal setting, a question of Matui on the almost finite--purely infinite dichotomy for minimal amenable ample groupoids. We prove that a second-countable Hausdorff minimal principal topologically amenable ample groupoid with Cantor unit space is strongly almost finite if and only if it admits an invariant probability measure; if no such measure exists, then it is purely infinite. We also prove that every bounded-degree uniformly Borel amenable F\o lner graph is Borel almost finite. The proofs combine the recent three-to-two comparison method of Glasner and Liu with the randomized F\o lner packing methods of Elek and Tim\'ar.
\end{abstract}
\vi
\vi
\tableofcontents

\section{Introduction}

\'Etale groupoids have been studied for several decades at the interface of topological dynamics and operator algebras. Their systematic use in $C^*$-algebra theory goes back to Renault's foundational monograph \cite{Renault}, and they have since become standard objects in the study of dynamical systems, orbit equivalence and operator algebras. A modern introduction to Hausdorff \'etale groupoids and their $C^*$-algebras is given by Sims \cite{Sims}.

In Cantor dynamics, ample groupoids form a particularly important class of \'etale groupoids and have been studied extensively since the work of Giordano--Putnam--Skau \cite{GPS} and Matui \cite{MatuiHomology,MatuiTFGSurvey}. They now form a well-established class connecting Cantor dynamics with the structure of groupoid $C^*$-algebras, topological full groups, groupoid homology, and related algebraic invariants.

Within this framework, almost finiteness and pure infiniteness have emerged as two fundamental and sharply contrasting forms of structure. Almost finiteness, introduced by Matui as a topological counterpart of hyperfiniteness, expresses the existence of finite compact-open approximations whose classes have arbitrarily small boundary \cite{MatuiHomology,MatuiTFGSurvey}. Pure infiniteness, also introduced by Matui in this context, encodes the opposite phenomenon: nonempty clopen sets admit paradoxical clopen equidecompositions \cite{MatuiSFT}. Both notions have become central in the study of topological full groups, groupoid homology, type semigroups, and groupoid $C^*$-algebras.

Matui asked this question explicitly in his 2025 lecture notes: does there exist a minimal ample amenable groupoid which is neither almost finite nor purely infinite? \cite[Problem~3.8]{MatuiNotes}. The same question is recorded as Problem~0.1 in Wu--Zhao--Ma \cite{WuZhaoMa}. The purpose of the present paper is to show that the answer is negative in the principal topologically amenable Cantor setting: there the almost finite--purely infinite dichotomy is complete.

Our main result is the following.

\begin{theorem}[Almost finite--purely infinite dichotomy]\label{thm:main-intro}
Let $\G$ be a second-countable Hausdorff minimal principal topologically amenable ample groupoid whose unit space $\G^{(0)}$ is a Cantor space. Then the following are equivalent:
\begin{enumerate}
\item $\G$ admits an invariant probability measure;
\item $\G$ is strongly almost finite.
\end{enumerate}
If these equivalent conditions hold, then $\G$ is almost finite in the sense of Matui. If they fail, then $\G$ is purely infinite.
\end{theorem}

The notion of strong almost finiteness used here will be defined in Section~2.

Suzuki asked two questions concerning the relation between almost finiteness and topological amenability \cite[Remark~3.7]{Suzuki}. The first asks whether every minimal almost finite \'etale Cantor groupoid is topologically amenable. The answer is negative: the first author constructed a minimal principal almost finite ample groupoid which is not topologically amenable; see \cite{ElekTimarStrong} and the references therein. The second asks whether a minimal topologically amenable \'etale Cantor groupoid admitting an invariant probability measure must be almost finite. Theorem~\ref{thm:main-intro} settles this second question in the principal case, and in fact yields strong almost finiteness.

The recent work of Glasner and Liu plays a central role in both the background and the proof of our result. Almost finiteness for free actions of amenable groups was introduced and developed by Kerr and subsequently related to dynamical comparison by Kerr--Szab\'o \cite{Kerr,KerrSzabo}. Glasner and Liu recently proved dynamical comparison for every action of a countably infinite amenable group on a zero-dimensional compact space \cite{GlasnerLiu}. Together with the Kerr--Szab\'o characterization, their theorem settles the almost-finiteness problem for free amenable Cantor actions: every such action is almost finite. Theorem~\ref{thm:main-intro} extends this from transformation groupoids to arbitrary principal topologically amenable ample Cantor groupoids. Equally importantly for the present paper, the Glasner--Liu proof contains a three-to-two cancellation mechanism which is one of the main ingredients of our argument. We combine it with the randomized F\o lner packing methods developed in \cite{ElekTimarUBA,ElekTimarStrong}.

The topological amenability assumption cannot be omitted. For instance, the free minimal Haar-measure preserving action of $F_2$ on its profinite completion gives a principal ample Cantor groupoid with an invariant probability measure which is neither almost finite nor purely infinite.

\subsection*{The Borel origin of the argument}

The Borel starting point of the present paper is Problem~3.2 in Marks's ICM survey \cite{MarksICM}. Let $\Gamma$ be a finitely generated amenable group with a finite symmetric generating set $S$, and let $\Gamma\curvearrowright X$ be a free Borel action on a standard Borel space. Marks asked whether, for every $\varepsilon>0$, the action admits a Borel tiling by finitely many $(\varepsilon,S)$-F\o lner shapes. In the graph language used here, this asks whether the Schreier graph of every such action is Borel almost finite.

Substantial special cases were known. Conley, Jackson, Marks, Seward and Tucker-Drob developed Borel asymptotic dimension and proved a general Borel F\o lner tiling theorem for free actions with finite asymptotic separation index \cite[Theorem~9.1]{CJMST}. Combined with their structural results on Borel asymptotic dimension, this gives a positive answer to Marks's problem for a large class of amenable groups, including groups of uniform local polynomial volume growth, all polycyclic groups, and the lamplighter group $\mathbb Z_2\wr\mathbb Z$ \cite{CJMST}. Thus the problem was known for groups of polynomial growth and already for several groups of exponential growth.

More recently, Greb\'ik, Marks, Rozho\v{n} and Shinko obtained a genuinely graph-theoretic result: every bounded-degree Borel graph of subexponential volume growth admits Borel F\o lner tilings \cite[Theorem~B]{GMRS}. In particular, the tiling phenomenon extends well beyond Schreier graphs coming from the group classes covered by the earlier asymptotic-dimension methods.

Motivated by this circle of questions, in \cite{ElekTimarUBA} we formulated the general graph-theoretic problem at the level of uniform Borel amenability: is every bounded-degree uniformly Borel amenable F\o lner graph Borel almost finite? This class contains the Schreier graphs in Marks's problem as well as bounded-degree Borel graphs of subexponential growth. Our second main theorem gives an affirmative answer.

\begin{theorem}[Borel almost finiteness]\label{thm:borel-intro}
Every bounded-degree uniformly Borel amenable F\o lner graph is Borel almost finite.
\end{theorem}

Thus, for every $\varepsilon>0$, the whole vertex space admits a Borel partition into finite sets of uniformly bounded diameter and relative boundary less than $\varepsilon$. In particular, Theorem~\ref{thm:borel-intro} settles Marks's Problem~3.2 and contains the Greb\'ik--Marks--Rozho\v{n}--Shinko theorem, as well as the preceding group-action cases, as special cases. We stress that the partitions obtained here are not asserted to be increasing, so no Borel hyperfiniteness conclusion is being claimed.

\subsection*{Applications}

Theorem~\ref{thm:main-intro} makes structural results previously requiring almost finiteness, pure infiniteness or comparison applicable throughout the principal amenable setting. In this subsection, let $\G$ satisfy the hypotheses of that theorem and put $X=\G^{(0)}$.

\paragraph{The invariant-measure case.}
If $\G$ admits an invariant probability measure, it is strongly almost finite and has dynamical comparison: whenever clopen sets $A,B\subseteq X$ satisfy $\mu(A)<\mu(B)$ for every invariant probability measure $\mu$, the set $A$ has a finite clopen partition whose pieces can be mapped by bisections into pairwise disjoint subsets of $B$.

The reduced algebra $C_r^*(\G)$ is stably finite and has stable rank one, real rank zero and strict comparison, by the results of Suzuki and Ara--B\"onicke--Bosa--Li \cite{Suzuki,ABBLStrict}. Ma--Wu's theorem gives $\mathcal Z$-stability \cite[Corollary~13.14]{MaWu}. Together with simplicity, nuclearity and the UCT, these results imply that $C_r^*(\G)$ is quasidiagonal, has nuclear dimension at most one, and is classified by its Elliott invariant, consisting of ordered $K$-theory, the class of the unit, traces and their pairing \cite{Tu,TWW,CETWW,MaWu}.

\paragraph{The case without an invariant measure.}
If $\G$ admits no invariant probability measure, every nonempty clopen subset of $X$ contains two disjoint copies of itself obtained by finite clopen equidecompositions along $\G$. Any clopen set can also be moved piecewise into any nonempty clopen set. The known groupoid pure-infiniteness criteria and Tu's UCT theorem then imply that $C_r^*(\G)$ is a unital UCT Kirchberg algebra \cite{BCS,BonickeLi,Tu}. It is therefore classified by its $K$-groups together with the class of the unit, by the Kirchberg--Phillips theorem \cite{Phillips}.

\paragraph{Free Cantor actions of nonamenable groups.}
In particular, let $\Gamma$ be a countable nonamenable group acting freely, minimally and topologically amenably on a Cantor space $X$. Such an action has no invariant probability measure: integrating its amenability witnesses against such a measure would imply amenability of $\Gamma$ \cite[Lemma~2.2]{GGKN}. Applying Theorem~\ref{thm:main-intro} to $\Gamma\ltimes X$ shows that every nonempty clopen subset of $X$ is paradoxical by piecewise $\Gamma$-translations. The action has dynamical comparison, and $C(X)\rtimes_r\Gamma$ is a UCT Kirchberg algebra.

R\o rdam--Sierakowski constructed, for every countable exact nonamenable group, a free minimal amenable Cantor action with a Kirchberg crossed product \cite[Theorem~6.11]{RordamSierakowski}. The conclusion here holds for \emph{every} such action. Comparison was already known for several classes of nonamenable groups, including nonamenable hyperbolic groups, by Gardella--Geffen--Kranz--Naryshkin \cite[Theorems~B and~C]{GGKN}. Our result establishes the free Cantor-space case of their Conjecture~D, without further restrictions on the acting group.

These conclusions also answer the Cantor-space cases of Problems~XXXII and~XXXIII in Schafhauser--Tikuisis--White's survey \emph{Nuclear $C^*$-algebras: 99 problems} \cite[Problems~XXXII and~XXXIII]{STW99}. Problem~XXXII asks for $\mathcal Z$-stability of crossed products of free minimal actions: part~(1) concerns countably infinite amenable groups acting on finite-dimensional compact metrizable spaces, and part~(2) concerns countable exact nonamenable groups acting topologically amenably on compact metrizable spaces. Problem~XXXIII asks for dynamical comparison for all free minimal topologically amenable actions of countable groups on compact metrizable spaces. For nonamenable groups, our theorem establishes both assertions on Cantor spaces and gives a UCT Kirchberg crossed product, not merely $\mathcal Z$-stability. For amenable groups, the Cantor-space conclusions already follow from Glasner--Liu \cite{GlasnerLiu}, together with Kerr's and Kerr--Szab\'o's results \cite{Kerr,KerrSzabo}.

\paragraph{Topological full groups.}
In both cases, the commutator subgroup of the topological full group $[[\G]]$ is simple, and the index map $I:[[\G]]\to H_1(\G)$ is surjective, by Matui's results \cite{MatuiHomology,MatuiSFT,MatuiTFGSurvey}. Since dynamical comparison holds, Xin Li's theorem \cite[Corollary~6.14]{Li} gives the exact sequence
\[
 H_0(\G)\otimes\mathbb Z/2\mathbb Z
 \longrightarrow [[\G]]_{\mathrm{ab}}
 \xrightarrow{\,I\,} H_1(\G)
 \longrightarrow 0.
\]
Here $H_i(\G)$ denotes groupoid homology with integer coefficients and $[[\G]]_{\mathrm{ab}}$ denotes the abelianization. Thus Matui's AH-conjecture holds for every groupoid in Theorem~\ref{thm:main-intro}.

Although our main theorem is stated for groupoids, the proof is organized almost entirely in graph-theoretic language. Section~2 gives a dictionary between principal ample Cantor groupoids and bounded-degree clopen graphs. In the principal setting, every compact-open bisection is uniquely determined by the partial homeomorphism that it induces on the unit space; after making this identification once, we work with the corresponding clopen orbit maps. Elementary subgroupoids become clopen tilings, and topological amenability becomes the corresponding finite graph condition. After that translation, the Borel and clopen arguments are presented as graph arguments, and no familiarity with groupoid formalism is needed to follow the main combinatorics. This organization is intended in particular for readers coming from Borel graph combinatorics and countable equivalence relations.

\section{Preliminaries: from groupoids to clopen graphs}\label{sec:prelim}

\paragraph{Clopen graphs and stable actions.}
Before introducing ample groupoids, we would like to take a brief detour through clopen graphs. These are the Cantor-space counterparts of Borel graphs, which are fundamental objects in Borel combinatorics. As we shall see, the ample groupoids studied in this paper can be described in the language of clopen graphs, allowing us to approach them from a familiar combinatorial viewpoint.

An action of a countable group $\Gamma$ on a topological space $X$ by homeomorphisms is \emph{stable} if $\Fix(g)=\{x\in X:gx=x\}$ is clopen for every $g\in\Gamma$. Continuity alone does not suffice: the shift on $\{0,1\}^{\mathbb Z}$ fixes exactly the two constant configurations, and this two-point set is not open.

The stability condition has a simple topological meaning. Give the space $\Sub(\Gamma)$ of subgroups the topology inherited from $\{0,1\}^{\Gamma}$ through their indicator functions. A continuous action is stable if and only if the stabilizer map $x\mapsto\Gamma_x=\{g\in\Gamma:gx=x\}$ is continuous.

\begin{definition}\label{def:clopen-graph}
A \emph{clopen graph} on a Cantor space $X$ is a simple graph $K$ admitting a presentation by a stable action $\Gamma\curvearrowright X$ of a countable group and a countable symmetric generating set $S$ of $\Gamma$, with
\[
 x\sim_K y
 \quad\Longleftrightarrow\quad
 x\ne y\ \text{ and }\ y=sx\ \text{ for some }s\in S.
\]
Thus $K$ is the underlying simple Schreier graph of the action, and its connected components are the $\Gamma$-orbits. If $S$ is finite, then $K$ is a \emph{bounded-degree clopen graph}, with degree at most $|S|$.
\end{definition}

Suppose that $S$ is finite. For an integer $r\geq0$, the \emph{labelled radius-$r$ neighbourhood} of $x$ has root $x$, all vertices at $K$-distance at most $r$ from $x$, and every directed edge $y\to sy$, labelled by $s\in S$, with both endpoints in this ball. Loops and distinct labels are retained. Two neighbourhoods have the same \emph{type} when an isomorphism preserves roots, directions and labels.

For each radius $r$, the points with any prescribed neighbourhood type form a clopen set. Indeed, let $u,v$ range over words in $S$ of length at most $r$, including the empty word. All vertices are of the form $ux$; vertex identifications $ux=vx$ and labelled edges $sux=vx$ are determined by membership in $\Fix(v^{-1}u)$ and $\Fix(v^{-1}su)$. These finitely many sets are clopen by stability, so their membership patterns give a finite clopen partition; each type corresponds to a union of its parts. For countable $S$, the same argument applies to each finite symmetric subset.

\paragraph{Minimality.}
The stable action $\Gamma\curvearrowright X$ is \emph{minimal} if every orbit $\Gamma x$ is dense in $X$. Equivalently, every connected component of $K$ is dense in its vertex space.

\paragraph{Topological amenability.}
Write $R_\Gamma=\{(x,y)\in X\times X:y\in\Gamma x\}$ for the orbit relation. A probability distribution on $\Gamma x$ is a function $p_x:\Gamma x\to[0,1]$ with $\sum_{z\in\Gamma x}p_x(z)=1$. A family $(p_x)_{x\in X}$ is \emph{continuous} if, for every $g\in\Gamma$, the function $x\mapsto p_x(gx)$ is continuous. Thus continuity is measured along the orbit coordinates supplied by the action.

\begin{definition}\label{def:orbit-amenability}
The orbit relation $R_\Gamma$ is \emph{topologically amenable} if, for every finite symmetric set $F\subseteq\Gamma$ and every $\varepsilon>0$, there are a finite set $D\subseteq\Gamma$ and a continuous family of probability distributions $(p_x)_{x\in X}$ on the orbits such that
\[
 \operatorname{supp}p_x\subseteq Dx=\{dx:d\in D\}\qquad(x\in X)
\]
and, for every $x\in X$ and every $g\in F$,
\begin{equation}\label{eq:orbit-variation}
 \|p_x-p_{gx}\|_1
 :=\sum_{z\in\Gamma x}|p_x(z)-p_{gx}(z)|<\varepsilon.
\end{equation}
\end{definition}

Both distributions in \eqref{eq:orbit-variation} are on the same orbit, so their masses are compared at the same points. These are distributions on \emph{orbit points}, not on elements of the abstract group $\Gamma$. The set $D$ is independent of $x$; it need not lie in the subgroup generated by $F$.

When $S$ is finite, it suffices to require \eqref{eq:orbit-variation} for $g\in S$. Indeed, if all elements of a finite set $F$ have $S$-length at most $L\geq1$, the generator estimate with $\varepsilon/L$ gives the required estimate by the triangle inequality along words. Also, $D$ lies in an $S$-ball of some radius $R$, so every $p_x$ is supported in the radius-$R$ neighbourhood of $x$ in $K$. This is the bounded-degree formulation of topological amenability in terms of continuous distributions on orbit graphs; see \cite[Section~10.3]{ElekTimarStrong}.

The same clopen graph can be presented by different stable actions, but its minimality and the topological amenability of its orbit relation do not depend on the chosen presentation \cite{RenaultBorel}.


\paragraph{Groupoids and bisections.}
A \emph{groupoid} $\G$ consists of arrows between the points of a \emph{unit space} $X=\G^{(0)}$. An arrow $g$ has source $s(g)$ and range $r(g)$, and the product $gh\in \G$ is defined when $s(g)=r(h)$. Multiplication is associative whenever the products are defined. Each $x\in X$ has a unit arrow $1_x$, with $1_{r(g)}g=g1_{s(g)}=g$, and every arrow has an inverse satisfying $g^{-1}g=1_{s(g)}$ and $gg^{-1}=1_{r(g)}$. A \emph{topological groupoid} carries a topology for which source, range, inversion and multiplication are continuous; multiplication is considered on
\[
\G^{(2)}=\{(g,h)\in\G\times\G:s(g)=r(h)\}
\]
with the subspace topology, and $X$ is identified with the subspace of unit arrows. We assume that $\G$ is Hausdorff and second countable and that $X$ is a Cantor space.

The groupoid is \emph{\'etale} if $s:\G\to X$ is a \emph{local homeomorphism}: every arrow $g\in\G$ has an open neighbourhood $V\subseteq\G$ such that $s(V)$ is open in $X$ and the restriction $s|_V:V\to s(V)$ is a homeomorphism. Equivalently, $r$ is a local homeomorphism. A \emph{bisection} is a subset $U\subseteq\G$ on which both $s$ and $r$ are injective. Its \emph{domain} is $s(U)$ and its \emph{range} is $r(U)$. For an open bisection, the restrictions of $s$ and $r$ are homeomorphisms onto open subsets of $X$, giving the partial homeomorphism
\[
\varphi_U=r\circ(s|_U)^{-1}:s(U)\longrightarrow r(U).
\]
The groupoid is \emph{ample} if it is \'etale and its unit space is zero-dimensional. Equivalently, compact-open bisections form a basis for its topology. In this case the source and range of a compact-open bisection are clopen in $X$. Finally, $\G$ is \emph{principal} if $s(g)=r(g)$ implies $g=1_{s(g)}$. From this point on, every groupoid considered in this paper is Hausdorff, second countable, ample and principal, with Cantor unit space.

The domain and range of a bisection need not be disjoint. In our setting, however, every compact-open bisection has a finite compact-open partition whose pieces either consist of units or have disjoint source and range. Indeed, $X$ is clopen in $\G$, and the non-unit part of the bisection is compact. At each of its arrows, principality gives two distinct endpoints; choose disjoint clopen neighbourhoods of them and shrink the bisection accordingly. A finite cover and disjoint refinement give the assertion.

\paragraph{The equivalence relation and a generating action.}
Principality implies that between two points there is at most one arrow. Thus $(r,s)$ identifies the underlying set of $\G$ with the equivalence relation
\[
R_{\G}=\{(y,x):\text{there is an arrow from }x\text{ to }y\}.
\]
The orbit of $x$ is its equivalence class. There are only countably many compact-open bisections: each is a finite union of members of a countable compact-open basis. Hence $R_{\G}$ is a countable Borel equivalence relation, being a countable union of graphs of partial homeomorphisms. The topology on $\G$ is generally \emph{not} the subspace topology inherited from $X\times X$. Each compact-open bisection $U$ provides a \emph{chart}: the homeomorphism $s|_U:U\to s(U)$ identifies the arrows of $U$ with their source points in the clopen set $s(U)\subseteq X$. These bisections form a basis for the topology of $\G$.

Enumerate all compact-open bisections with disjoint source and range. Each associated map extends to an involution of $X$ by using its inverse on the range and the identity elsewhere. The resulting involutions $\sigma_1,\sigma_2,\ldots$ define an action
\[
\Gamma_{\G}=\mathop{*}_{n\geq1}C_2\ \curvearrowright\ X
\]
whose orbits are exactly the $\G$-orbits. This is the clopen counterpart of the Feldman--Moore theorem, which represents a countable Borel equivalence relation by a countable group of Borel automorphisms \cite[Theorem~1]{FeldmanMoore}; in its involutive formulation the acting group may likewise be taken to be a countable free product of copies of $C_2$.

Observe that the action constructed above is stable. Indeed, each involution $\sigma_i$ comes with a continuous assignment $x\mapsto a_i(x)$ of an arrow from $x$ to $\sigma_i x$: use the defining bisection on its domain, the inverse bisection on its range, and unit arrows elsewhere. Continuity follows from the bisection charts and the fact that these three subsets of $X$ are clopen. For any $g\in\Gamma_{\G}$, choose a word in the generators representing $g$. Following the corresponding arrows and multiplying them gives an arrow $a_g(x)$ from $x$ to $gx$. This defines a continuous map $a_g:X\to\G$, by continuity of the generator assignments and groupoid multiplication. Principality implies that $gx=x$ exactly when $a_g(x)$ is a unit arrow.

We now verify that the unit arrows form a clopen subset of the arrow space $\G$. They form a closed set because $X$ is compact and $\G$ is Hausdorff. For openness, fix $x\in X$ and choose an open bisection $U$ containing $1_x$. The set $V=\{y\in X:1_y\in U\}$ is an open neighbourhood of $x$ in $X$, since $X$ carries the subspace topology. If $h\in U$ and $s(h)\in V$, then both $h$ and $1_{s(h)}$ lie in $U$ and have the same source. The injectivity of $s|_U$ therefore gives $h=1_{s(h)}$. Consequently,
\[
U\cap s^{-1}(V)=\{1_y:y\in V\}\subseteq X.
\]
The left-hand side is open in $\G$ and contains $1_x$. Every unit arrow thus has an open neighbourhood consisting only of unit arrows, which proves that $X$ is open in $\G$.

Since $a_g:X\to\G$ is continuous, it follows that
\[
\Fix(g)=\{x\in X:a_g(x)\in X\}=a_g^{-1}(X)
\]
is clopen in $X$. In concrete terms, if $gx=x$, then $a_g(x)$ is a unit arrow, and continuity ensures that $a_g(y)$ is still a unit arrow for all $y$ in some neighbourhood of $x$. Hence $gy=y$ throughout that neighbourhood. This holds for every $g\in\Gamma_{\G}$, proving stability.

Conversely, a stable action $\Gamma\curvearrowright X$ of a countable group on a Cantor space defines a Hausdorff ample principal groupoid on its orbit relation by using the sets $\{(gx,x):x\in U\}$, with $U$ clopen, as charts. Indeed, the charts given by $g,h\in\Gamma$ on clopen sets $U,V\subseteq X$ describe the same arrow at a source point $x\in U\cap V$ precisely when $gx=hx$. The set of such source points is
\[
\{x\in U\cap V:gx=hx\}=U\cap V\cap\Fix(h^{-1}g),
\]
which is clopen by stability. Hence the intersection of the two charts corresponds to a clopen subset of each of their domains. For the action $\Gamma_{\G}\curvearrowright X$ constructed above these charts recover the original topology, not just the original orbits.

We fix the action $\Gamma_{\G}\curvearrowright X$ constructed above once and for all. A \emph{clopen orbit map} is a homeomorphism between clopen subsets which, on each piece of a finite clopen partition of its domain, agrees with an element of $\Gamma_{\G}$. Because $\G$ is principal, a compact-open bisection is uniquely determined by the partial homeomorphism that it induces on the unit space. Conversely, every clopen orbit map is induced by a unique compact-open bisection: refine its domain into finitely many clopen pieces on which the map agrees with one of the involutions above, and take the corresponding compact-open arrow sets. Thus, in the principal setting of this paper, compact-open bisections and clopen orbit maps carry exactly the same information. From now on we use only the clopen orbit map terminology in the graph-theoretic arguments.

\paragraph{Bounded-degree clopen subgraphs.}
We now apply Definition~\ref{def:clopen-graph} to the stable action $\Gamma_{\G}\curvearrowright X$. For a finite symmetric set $S\subseteq\Gamma_{\G}$, let $H_S$ be the graph on $X$ in which distinct points $x,y$ are adjacent if $y=sx$ for some $s\in S$. The restricted action $\langle S\rangle\curvearrowright X$ is stable, so $H_S$ is a bounded-degree clopen graph in the sense of that definition. We call these graphs the \emph{bounded-degree clopen subgraphs} of the orbit structure of $\G$, and call $S$ a \emph{finite family of clopen generators}. Their connected components are the $\langle S\rangle$-orbits, which may be proper subsets of the full $\Gamma_{\G}$-orbits.

Equivalently, one may specify such a graph by finitely many clopen orbit maps and their inverses, omitting loops. To obtain a presentation $H_S$, partition the corresponding compact-open bisections as above, discard their unit pieces, and extend each remaining piece with disjoint domain and range to one of the involutions $\sigma_i$. These finitely many involutions give exactly the prescribed edges. Conversely, the maps $x\mapsto sx$, $s\in S$, are clopen orbit maps. In groupoid terms, the arrows corresponding to the oriented edges form a symmetric compact-open subset of $\G\setminus X$. Every such subset has a finite cover by compact-open bisections and therefore gives a bounded-degree clopen subgraph.

For a graph $K$ write $Kx$ for the neighbours of $x$, $B_K(x,r)$ for its radius-$r$ ball and $N_K(F)=F\cup\bigcup_{x\in F}Kx$. For a finite $D\subseteq\Gamma_{\G}$ put $DF=\{dx:d\in D,x\in F\}$. The inclusion $F\subseteq Dx$ means that every point of $F$ is obtained from $x$ using an element of the same finite set $D$. This is a bound in the full orbit, not necessarily a bound in the metric of an independently fixed graph $K$.

\paragraph{Minimality and invariant measures.}
The groupoid is minimal exactly when every $\Gamma_{\G}$-orbit is dense. An individual graph $H_S$ need not have dense components. A Borel probability measure $\mu$ is $\G$-invariant exactly when it is $\Gamma_{\G}$-invariant, or equivalently when
\[
\mu(gA)=\mu(A)\qquad(g\in\Gamma_{\G},\ A\subseteq X\text{ clopen}).
\]
This is also equivalent to preservation by every clopen orbit map. We denote the set of these measures by $\M$. Thus minimality and invariant measures refer to the full action, not to a single choice of finite generators.

\paragraph{Topological amenability.}
The groupoid $\G$ is topologically amenable exactly when the orbit relation of $\Gamma_{\G}\curvearrowright X$ satisfies Definition~\ref{def:orbit-amenability}. For $H_S$, we apply the same definition to the restricted action $\langle S\rangle\curvearrowright X$. We have
\begin{equation}\label{eq:amenability-locality}
\begin{gathered}
\G\text{ is topologically amenable}\\
\Longleftrightarrow\quad
H_S\text{ is topologically amenable for every finite symmetric }S\subseteq\Gamma_{\G}.
\end{gathered}
\end{equation}
Indeed, amenability passes to open subgroupoids, and every compact set of arrows lies in an open subgroupoid generated by finitely many compact-open bisections; see \cite[Proposition~5.1.1]{ADR} and \cite[Lemma~3.3]{AustadBonicke}.

For later use, the probabilities in Definition~\ref{def:orbit-amenability} may be chosen in the form
\begin{equation}\label{eq:orbit-amenability}
 p_x=\frac1n\sum_{j=1}^{n}\delta_{g_{P,j}x}\quad(x\in P\in\mathcal P),
 \qquad g_{P,j}\in D,
\end{equation}
where $\mathcal P$ is a finite clopen partition, $n\geq1$ and $D\subseteq\Gamma_{\G}$ is finite. Repetitions are allowed, so all weights have the form $k/n$. To obtain this form, start with a smaller error. Stability gives a finite clopen partition on which the equalities $dx=d'x$, $d,d'\in D$, are fixed; on a finer clopen partition, approximate the resulting continuous probability vectors by rational ones with a common denominator. For $H_S$, we use $D\subseteq\langle S\rangle$, so the supports have uniformly bounded $H_S$-radius. For the full orbit relation, $D$ need not lie in $\langle S\rangle$.

\paragraph{Finite amplifications and clopen tilings.}
For $m\geq1$ let $[m]=\{1,\ldots,m\}$. On $X\times[m]$, allowed local maps have the form $(x,i)\mapsto(gx,j)$ on finitely many clopen pieces. The orbit of $(x,i)$ is $\Gamma_{\G} x\times[m]$. Stability and the local neighbourhood description remain valid. Amenability also passes to this finite amplification: distribute each orbit probability uniformly over the sheets. The resulting probabilities agree under a sheet change $(x,i)\mapsto(x,j)$, while for a clopen orbit map $(x,i)\mapsto(gx,i)$ their $\ell^1$-variation is the same as that of the original probabilities at $x$ and $gx$; the same locality statement then applies to its bounded-degree clopen subgraphs.

A \emph{clopen packing} is a finite family of towers $\{\tau_1,\tau_2,\cdots\tau_m\}$. The \emph{tower} $\tau_i$  consists of a clopen base $B_i\subset X$ and finitely many clopen orbit maps
\[
\varphi_{i,j}:B_i\longrightarrow X\quad(1\leq j\leq h_i),
\qquad \varphi_{i,1}=\id_{B_i},
\]
whose ranges, over all $i,j$, are pairwise disjoint. The sets $\varphi_{i,j}(B_i)$ are called the \emph{levels} of the tower. Its tiles are the finite sets $T_{i,b}=\{\varphi_{i,j}(b):1\leq j\leq h_i\}$, $b\in B_i$. The union of all tiles is clopen; we denote it by $[\mathcal T]$. A packing covering $X$ is a \emph{clopen tiling}. The equivalence relation consisting of all pairs within a tile is a compact-open elementary subgroupoid, and every such subgroupoid has this tower description. To see the converse, describe its compact arrow space by finitely many clopen orbit maps, choose a finite clopen colouring separating points in the same class, and use the least colour in each class as its base; finite clopen refinement then lists the levels. See also \cite{MatuiHomology}.

Each tower is a clopen family of tiles of one fixed combinatorial type, parametrized by its base. For a fixed $i$, the point $b\in B_i$ parametrizes the tile $T_{i,b}$, while the level $\varphi_{i,j}(B_i)$ records the same position $j$ across all tiles in that tower.

In particular a clopen packing has uniformly bounded tile sizes. There is a common finite symmetric $D\subseteq\Gamma_{\G}$ such that
\begin{equation}\label{eq:orbit-bounded-tiles}
T\subseteq Dx\qquad(x\in T,\ T\in\mathcal T).
\end{equation}
The same $D$ can be used for a finite family of packings. The tiles therefore have bounded diameter in some bounded-degree clopen graph, but not necessarily in a previously prescribed graph $K$. The tower description makes all finite cardinality and incidence tests clopen.

\paragraph{F\o lner sets and almost finiteness.}
For a bounded-degree clopen subgraph $K$ and a nonempty finite set $F$ contained in one \emph{full} orbit, put
\[
\partial_KF=\{x\in F:Kx\not\subseteq F\}.
\]
We call $F$ $\delta$-F\o lner for $K$ if $|\partial_KF|<\delta|F|$. The set $F$ may meet several $K$-components. The \emph{F\o lner property of the orbit structure} means that such an $F$ exists for every $K$ and every $\delta>0$. For this existence statement one may equivalently require $F$ to lie in one $K$-component: split it into its component intersections and choose one with the smallest relative boundary. Minimality implies that there is a finite $D\subseteq\Gamma_{\G}$ such that, for every $x\in X$, one of these configurations is contained in $Dx$. It does not assert their occurrence in uniformly bounded $K$-balls.

The orbit structure is \emph{almost finite} if, for every $K$ and $\varepsilon>0$, it has a clopen tiling all of whose tiles are $\varepsilon$-F\o lner for $K$. This is exactly Matui's almost finiteness. Indeed, if $d\geq1$ bounds the degree of $K$, then
\[
|N_K(F)\setminus F|\leq d|\partial_KF|,
\qquad |\partial_KF|\leq d|N_K(F)\setminus F|.
\]
Every compact set of arrows is described by finitely many clopen orbit maps, so these inequalities translate the two boundary conventions in both directions. The finite orbit bound \eqref{eq:orbit-bounded-tiles}, rather than a bound in $K$, is the appropriate bound on the tiles.

For a clopen tiling $\mathcal E$ let
\[
D_K(\mathcal E)=\{x:Kx\not\subseteq\mathcal E(x)\},
\]
where $\mathcal E(x)$ is the tile containing $x$. We call the structure \emph{strongly almost finite} if, for every $K$ and $\tau,\eta>0$, there are clopen tilings $\mathcal E_1,\ldots,\mathcal E_m$, all of whose tiles are $\tau$-F\o lner for $K$, with
\begin{equation}\label{eq:strong-af-definition}
\frac1m\sum_{i=1}^{m}\mathbf1_{D_K(\mathcal E_i)}(x)<\eta\qquad(x\in X).
\end{equation}
Thus every point is an interior point in an arbitrarily large proportion of finitely many F\o lner tilings.

\paragraph{Comparison, the type semigroup and the full group.}
For clopen subsets $A,B$ of finite amplifications, write $A\preceqcl B$ if there is an injection $A\to B$ given by allowed orbit maps on finitely many clopen pieces. We call $A$ and $B$ \emph{equidecomposable} if there are finite clopen partitions
\[
A=\bigsqcup_{i=1}^{q}A_i,\qquad B=\bigsqcup_{i=1}^{q}B_i,
\]
and, for each $i$, a homeomorphism $\varphi_i:A_i\to B_i$ given by an allowed orbit map. These partitions and maps form an \emph{equidecomposition} of $A$ and $B$.

The \emph{clopen type semigroup} consists of equidecomposition classes $[A]$ of these sets, with addition given by disjoint union after relabelling sheets. Its algebraic preorder satisfies $[A]\leq[B]$ exactly when $A\preceqcl B$. It is the type semigroup of $\G$, since the allowed maps are exactly its compact-open local maps. Pure infiniteness means
\[
2[A]\leq[A]\qquad(\varnothing\ne A\subseteq X\text{ clopen}),
\]
that is, every nonempty clopen set contains two disjoint clopen orbit copies of itself. Finally, the homeomorphisms of $X$ which are piecewise elements of $\Gamma_{\G}$ on a finite clopen partition form the topological full group $[[\G]]$.

The translation from groupoid language to graph language is now fixed. In the clopen proofs in Sections~\ref{sec:clopen} and~\ref{sec:pure}, we use the stable orbit structure, its clopen orbit maps and its bounded-degree subgraphs. All maps used in these clopen constructions are clopen orbit maps and therefore send each point to a point in the same $\Gamma_{\G}$-orbit; the chosen bounded-degree subgraph specifies which boundary is being measured.

\section{The Borel case}\label{sec:borel}

This section is completely self-contained. Let $G$ be a bounded-degree Borel graph on a standard Borel space $X$. We write $d_G$ for graph distance and
\[
B_r(x)=\{y\in X:d_G(x,y)\le r\}.
\]
For a finite set $F\subseteq X$, its inner boundary is
\[
\partial F=\{x\in F:\text{ there exists }y\notin F\text{ with }xGy\}.
\]
We call $F$ $\delta$-F\o lner if
\[
|\partial F|<\delta |F|.
\]

\begin{definition}[Uniform Borel amenability]
The graph $G$ is uniformly Borel amenable if for every $\alpha>0$ there are $R<\infty$ and a Borel family of probability measures $x\mapsto\lambda_x$ on $X$ such that
\[
\operatorname{supp}(\lambda_x)\subseteq B_R(x)
\]
and
\[
xGy\quad\Longrightarrow\quad \|\lambda_x-\lambda_y\|_1<\alpha.
\]
\end{definition}

\begin{definition}[F\o lner property]
The graph $G$ has the F\o lner property if for every $\delta>0$ there is $r<\infty$ such that every ball $B_r(x)$ contains a nonempty finite $\delta$-F\o lner set.
\end{definition}

For $m\ge1$, put
\[
[m]=\{1,\ldots,m\}.
\]
For $A\subseteq X$, write
\[
mA=A\times[m].
\]
If $A\subseteq X\times\{1,\ldots,m\}$ and $B\subseteq X\times\{1,\ldots,n\}$ are Borel, we write $A\preceqG B$ if there are a Borel injection $\Phi:A\to B$ and $D<\infty$ such that
\[
d_G\bigl(\pi_X(z),\pi_X(\Phi(z))\bigr)\le D\qquad(z\in A),
\]
where $\pi_X(x,k)=x$. All distance bounds in this section refer to the fixed graph $G$.

A \emph{Borel packing} consists of a Borel set $Y\subseteq X$ and a finite Borel equivalence relation on $Y$ whose classes have uniformly bounded $G$-diameter. Its classes are called tiles, and we write $[\mathcal P]=Y$ for the union of the tiles. A packing is $\delta$-F\o lner if all of its tiles are $\delta$-F\o lner.

When $G$ is uniformly Borel amenable and has the F\o lner property, we use two results from Elek--Tim\'ar \cite{ElekTimarUBA}. The first is the randomized almost-finiteness conclusion of \cite[Theorem~5]{ElekTimarUBA}; the second is \cite[Proposition~3.1]{ElekTimarUBA}.

\medskip
\noindent\textbf{Packing input I.}
For every $\theta>0$ there are Borel packings
\[
\mathcal P_1,\ldots,\mathcal P_m
\]
with a common diameter bound such that every $\mathcal P_i$ is $\theta$-F\o lner and
\[
\bigl|\{i:x\in[\mathcal P_i]\}\bigr|>(1-\theta)m
\qquad(x\in X).
\]

\medskip
\noindent\textbf{Packing input II.}
For every $t>0$ there are a $t$-F\o lner Borel packing $\mathcal T$ of uniformly bounded diameter and a number $\rho>0$ such that, for every nonempty finite $\rho$-F\o lner set $F$,
\begin{equation}
\left|\bigcup\{T\in\mathcal T:T\subseteq F\}\right|\ge (1-t)|F|.
\label{eq:borel-packing-II}
\end{equation}

\begin{theorem}[Borel almost finiteness]\label{thm:borel-main}
Assume that $G$ is uniformly Borel amenable and has the F\o lner property. Then for every $\varepsilon>0$ there is a Borel partition of $X$ into finite sets of uniformly bounded diameter such that every part $T$ satisfies
\[
|\partial T|<\varepsilon |T|.
\]
\end{theorem}

The two external ingredients used in the proof have different roles. From Elek--Tim\'ar \cite{ElekTimarUBA} we use the two packing statements above, in particular the large F\o lner packing of Packing input II. From Glasner--Liu \cite{GlasnerLiu} we use not their almost-finiteness theorem, but the finite matching mechanism developed in Sections~3--4 and used there to prove the three-to-two theorem in Section~5: the finite elimination of short augmenting paths and, ultimately, the passage
\[
3A\preceq_G 2B\quad\Longrightarrow\quad A\preceq_G B.
\]
We adapt this mechanism to the Borel finite-chart setting. Combined with the Elek--Tim\'ar packing, it yields the uniform F\o lner comparison theorem below, which turns a uniform local surplus of reserve points on every sufficiently F\o lner finite set into a global bounded-displacement Borel injection. This is the additional step that allows the uncovered set of the packing to be absorbed everywhere, rather than only modulo the exceptional compressible part.

\subsection{The three-to-two lemma after Glasner and Liu}

\begin{lemma}[Borel three-to-two]\label{lem:borel-three-two}
Assume that $G$ is uniformly Borel amenable. If $A,B\subseteq X$ are Borel and
\[
3A\preceqG 2B,
\]
then
\[
A\preceqG B.
\]
Moreover, if $J:3A\to2B$ is the given injection, the concluding injection sends each $a\in A$ to the underlying $B$-coordinate of one of
\[
J(a,1),\quad J(a,2),\quad J(a,3).
\]
\end{lemma}

\begin{proof}
If $A=\varnothing$, there is nothing to prove. Fix a Borel injection
\[
J:3A\longrightarrow2B
\]
and $D<\infty$ such that, whenever
\[
J(a,i)=(b,j),
\]
we have
\begin{equation}
d_G(a,b)\le D.
\label{eq:borel-three-two-displacement}
\end{equation}
Put $Y=2B$ and $F=J(3A)$. Define a Borel permutation $p$ of $Y$ by
\[
p(J(a,1))=J(a,2),\qquad
p(J(a,2))=J(a,3),\qquad
p(J(a,3))=J(a,1)
\]
for each $a\in A$, and by $p(y)=y$ for $y\in Y\setminus F$. Define also
\[
q(b,1)=(b,2),\qquad q(b,2)=(b,1).
\]

Fix a Borel linear order $<$ on $Y$. Let $V\subseteq Y$ contain the $<$-least point of every $p$-orbit, and let
\[
\pi:Y\longrightarrow V
\]
map a point to the representative of its $p$-orbit. We form a Borel multigraph $H$ with vertex set
\[
V(H)=V
\]
and edge set
\[
E(H)=\{e_b:b\in B\}.
\]
For every $b\in B$, the edge $e_b$ has endpoints
\[
\pi(b,1),\qquad \pi(b,2).
\]
Thus $e_b$ is a loop if these two endpoints coincide, and different $b$'s may give parallel edges. As usual, a loop contributes two to the degree of its vertex. Consequently, a vertex corresponding to a three-point $p$-orbit has degree three, while a vertex corresponding to a fixed point of $p$ has degree one.

The degree-three vertices are precisely the vertices
\[
v_a=\pi(J(a,1))=\pi(J(a,2))=\pi(J(a,3)),
\qquad a\in A.
\]
Thus they are naturally indexed by the points of $A$. The degree-one vertices come from the fixed points of $p$ in $Y\setminus F$; they are auxiliary vertices of $H$, and no edge has to be selected at them.

It is therefore enough to select, in a Borel way, one edge incident to each degree-three vertex $v_a$, with the requirement that no non-loop edge is selected at two distinct degree-three endpoints. Such a selection defines the desired bounded-displacement injection
\[
I:A\longrightarrow B.
\]
Namely, if the edge selected at $v_a$ is $e_b$, define
\[
I(a)=b.
\]
Since $e_b$ is incident to $v_a$, one of the two points $(b,1),(b,2)$ belongs to the $p$-orbit
\[
\{J(a,1),J(a,2),J(a,3)\}.
\]
Hence for some $i\in\{1,2,3\}$ and $j\in\{1,2\}$ we have
\[
J(a,i)=(b,j).
\]
In particular, $I(a)=b$ is the underlying $B$-coordinate of one of the three original $J$-images of $a$. If $e_b$ is a loop at $v_a$, this still simply gives $I(a)=b$; no further choice is involved.

If $I(a)=I(a')=b$ for two distinct points $a,a'\in A$, then the same non-loop edge $e_b$ has been selected at the two distinct degree-three vertices $v_a$ and $v_{a'}$, contrary to our requirement. Thus $I$ is injective. Finally, the displayed equality $J(a,i)=(b,j)$ and \eqref{eq:borel-three-two-displacement} give
\[
d_G(a,I(a))=d_G(a,b)\le D.
\]
Thus $I$ has the required bounded displacement.

From this point on, it remains only to construct the Borel edge selection just described. Use the fixed Borel order on $Y$ also to order the edges, by declaring $e_b<e_c$ whenever $(b,1)<(c,1)$.

We first treat the components containing a degree-one vertex. Let $L\subseteq V(H)$ be the Borel set of all degree-one vertices. A degree-one vertex selects no edge. If $v$ is a degree-three vertex in the $H$-saturation of $L$, select the least incident edge $\{v,w\}$ such that
\begin{equation}
d_H(w,L)=d_H(v,L)-1.
\label{eq:borel-leaf-direction}
\end{equation}
Thus the selected edge points one step along a shortest path from $v$ toward a degree-one vertex. Since $H$ is locally finite, the distance to the Borel set $L$ is Borel, so this selection is Borel. No non-loop edge is selected at both endpoints, because along a selected edge the distance to $L$ strictly decreases.

We next treat the components which contain no degree-one vertex but do contain a cycle. Such a component is three-regular. For $n\ge1$, let $X_n$ be the union of those components whose shortest cycle has length $n$, where a loop and a pair of parallel edges count as cycles of lengths one and two. The sets $X_n$ are Borel. Indeed, for each $m$ the set of vertices lying on an $m$-cycle is Borel, and its saturation under the connectedness relation of the locally finite Borel graph $H$ is Borel; $X_n$ is obtained by requiring the presence of an $n$-cycle and the absence of cycles of smaller length.

Fix $n$. On $X_n$, consider the Borel graph which joins two distinct vertices whenever their $H$-distance is at most $n$. This graph has bounded degree, so by the Kechris--Solecki--Todorcevic finite Borel colouring theorem \cite{KST} there is a Borel colouring
\[
\kappa_n:X_n\longrightarrow\{1,\ldots,r_n\}
\]
such that vertices of the same colour are at $H$-distance greater than $n$.

Let $C_n\subseteq X_n$ be the Borel set of vertices which lie on an $n$-cycle. In each connected component of $X_n$, let $c$ be the least colour assumed by $\kappa_n$ on $C_n$, and let $S_n$ consist of all vertices of $C_n$ having this colour. This is a Borel set, and every component of $X_n$ meets $S_n$. For each $x\in S_n$, choose, in the fixed Borel order, one $n$-cycle $C_x$ containing $x$. These cycles are pairwise vertex-disjoint. Indeed, if $C_x$ and $C_y$ met, then
\[
d_H(x,y)\le n,
\]
whereas $x$ and $y$ have the same colour in their component and hence must be at distance greater than $n$.

Let $D_n$ be the union of the vertices of the cycles $C_x$, $x\in S_n$. Orient each non-loop cycle $C_x$ in one of its two directions using the fixed Borel orders on vertices and edges, and let every vertex on it select its outgoing cycle edge; if $C_x$ is a loop, its unique vertex selects that loop. Every degree-three vertex $v\in X_n\setminus D_n$ selects the least incident edge $\{v,w\}$ for which
\[
d_H(w,D_n)=d_H(v,D_n)-1.
\]
Since $D_n$ meets every component of $X_n$, this is defined everywhere on $X_n\setminus D_n$. Again the selection is Borel, and no non-loop edge is selected at both endpoints: outside $D_n$ the distance to $D_n$ strictly decreases, while vertices of $D_n$ select only their prescribed cycle edges.

After these two constructions, the only components on which no edge has yet been selected are three-regular trees. We now use uniform Borel amenability.

Set
\[
\whX=X\times\{1,2\}.
\]
Define a bounded-degree Borel graph $\whG$ on $\whX$ by
\[
(x,i)\,\whG\,(y,i)\quad\Longleftrightarrow\quad xGy,
\]
and
\begin{equation}
(x,1)\,\whG\,(x,2)\qquad(x\in X).
\label{eq:borel-amplified-graph}
\end{equation}
Uniform Borel amenability passes from $G$ to $\whG$. More explicitly, if $x\mapsto\lambda_x$ is an amenability witness for $G$ with
\[
\operatorname{supp}(\lambda_x)\subseteq B_R^G(x),
\]
define
\begin{equation}
\widehat\lambda_{(x,i)}(y,j)=\frac12\lambda_x(y),
\qquad i,j\in\{1,2\}.
\label{eq:borel-amplified-probability}
\end{equation}
Then
\[
\operatorname{supp}(\widehat\lambda_{(x,i)})\subseteq B_{R+1}^{\whG}(x,i).
\]
The two measures at the ends of a vertical edge in \eqref{eq:borel-amplified-graph} are equal, while along a horizontal edge their $\ell^1$-distance is $\|\lambda_x-\lambda_y\|_1$.

The vertex set $V$ of $H$ is a Borel subset of $Y\subseteq\whX$. If two vertices $v,w\in V$ are joined by an edge of $H$, then
\begin{equation}
d_{\whG}(v,w)\le K_0:=4D+3.
\label{eq:borel-incidence-displacement}
\end{equation}
Indeed, let the edge be $e_b$. Its two endpoints are the representatives of the $p$-orbits containing $(b,1)$ and $(b,2)$, respectively. If $y,y'$ lie in the same three-point $p$-orbit, then by \eqref{eq:borel-three-two-displacement} their underlying $X$-coordinates are both within distance $D$ of the same point of $A$, and hence
\[
d_{\whG}(y,y')\le 2D+1.
\]
The two points $(b,1)$ and $(b,2)$ are at $\whG$-distance one, which proves \eqref{eq:borel-incidence-displacement}.

Apply uniform Borel amenability of $\whG$ with edge error smaller than $1/(4K_0)$. By the triangle inequality we obtain Borel probabilities $v\mapsto\lambda_v$ on $\whX$, supported in a common finite-radius ball about $v$, such that
\begin{equation}
d_{\whG}(v,w)\le K_0
\quad\Longrightarrow\quad
\|\lambda_v-\lambda_w\|_1<\frac14.
\label{eq:borel-ambient-variation}
\end{equation}

Let $E_H$ be the connectedness relation of $H$ restricted to the union of its three-regular tree components. Since $H$ is locally finite and Borel,
\[
v\,E_H\,w
\quad\Longleftrightarrow\quad
\text{there is an $H$-path from $v$ to $w$ of some finite length}
\]
is a Borel relation.

Fix one of these tree components $T$. For every point $z$ in the same $\whG$-component as $T$, let $r_T(z)$ be a vertex of $T$ which is closest to $z$ in the $\whG$-metric. If there is more than one closest vertex, choose the $<$-least one. Thus $r_T$ is simply the nearest-point map from the ambient $\whG$-component onto the tree $T$.

This choice is Borel. Indeed, starting from $z$, inspect the finite balls $B_n^{\whG}(z)$ for $n=0,1,2,\ldots$ until one first meets $T$, and then take the $<$-least point of that finite intersection. Since membership in the tree component $T=[v]_{E_H}$ is Borel, this defines the maps $r_T$ in a Borel way, uniformly over all three-regular tree components.

For $v\in T$, define a probability measure on the vertex set of $T$ by
\begin{equation}
\nu_v=(r_T)_*\lambda_v.
\label{eq:borel-tree-pushforward}
\end{equation}
If $v,w$ are adjacent in $T$, then \eqref{eq:borel-incidence-displacement}, \eqref{eq:borel-ambient-variation}, and contraction of the $\ell^1$-norm under pushforward give
\begin{equation}
\|\nu_v-\nu_w\|_1
\le \|\lambda_v-\lambda_w\|_1
<\frac14.
\label{eq:borel-tree-variation}
\end{equation}

Let $e=\{v,w\}$ be an edge of $T$. Removing $e$ leaves two components. Denote by $T_{v\to w}$ the component containing $w$ and put
\begin{equation}
f(v,e)=\nu_v(T_{v\to w}).
\label{eq:borel-half-tree-mass}
\end{equation}
The function $(v,e)\mapsto f(v,e)$ is Borel. Indeed, the relation
\[
u\in T_{v\to w}
\]
is Borel, since it means that $u$ and $w$ are joined by an $H$-path avoiding $e$. If the measures $\lambda_v$ in \eqref{eq:borel-ambient-variation} are supported in $B_R^{\whG}(v)$, then
\begin{equation}
f(v,e)=\sum_{z\in B_R^{\whG}(v)}
\lambda_v(z)\,
\mathbf 1_{\{r_T(z)\in T_{v\to w}\}}.
\label{eq:borel-half-tree-borel}
\end{equation}
The balls in this sum are finite with a uniform cardinality bound, so the right-hand side is Borel.

If $e_1,e_2,e_3$ are the three edges incident to $v$, the sets $T_{v\to w_i}$ are pairwise disjoint, and therefore
\begin{equation}
\sum_{i=1}^3 f(v,e_i)\le1.
\label{eq:borel-three-sector-sum}
\end{equation}
Hence
\begin{equation}
\min_{e\ni v}f(v,e)\le\frac13.
\label{eq:borel-three-sector-min}
\end{equation}
For an edge $e=\{v,w\}$ the two half-trees are complementary, so
\begin{equation}
\begin{aligned}
f(v,e)+f(w,e)
&=\nu_v(T_{v\to w})+\nu_w(T_{w\to v})\\
&\ge 1-\|\nu_v-\nu_w\|_1\\
&>\frac34.
\end{aligned}
\label{eq:borel-opposite-sector-sum}
\end{equation}
In particular the same edge cannot satisfy $f(v,e)\le1/3$ at both endpoints.

At every $v\in T$, choose the least incident edge, in the fixed Borel order, among those on which $f(v,e)$ is minimal. By \eqref{eq:borel-three-sector-min}--\eqref{eq:borel-opposite-sector-sum}, no edge is chosen at both endpoints. This completes the required edge selection at every degree-three vertex of $H$. The map $I$ defined above is therefore the required injection $A\to B$. By construction it sends $a$ to the underlying $B$-coordinate of one of $J(a,1),J(a,2),J(a,3)$.
\end{proof}

\subsection{Finite-stage augmentation}\label{subsec:borel-augmentation}

The following finite-stage matching lemma is the Borel form of the short-augmenting-path elimination mechanism from Glasner--Liu \cite[Section~3]{GlasnerLiu}.

Let $L$ and $R$ be standard Borel spaces, serving as the left and right vertex spaces. A finite-chart bipartite multigraph between $L$ and $R$ is given by finitely many Borel partial bijections
\[
\phi_i:D_i\longrightarrow R_i,
\qquad 1\le i\le d.
\]
The edge $(i,x)$ joins $x\in D_i$ to $\phi_i(x)$. Different values of $i$ may give parallel edges. A matching is a Borel set of edges with no repeated endpoint. We write $\dom M\subseteq L$ and $\ran M\subseteq R$ for the matched vertices; on $\ran M$, the notation $M^{-1}(y)$ means the unique left vertex matched to $y$.

An augmenting path for $M$ is a finite path
\[
x_0,y_0,x_1,y_1,\ldots,x_k,y_k,
\]
where $x_i\in L$, $y_i\in R$, the endpoints $x_0$ and $y_k$ are unmatched, the edges $x_i y_i$ do not belong to $M$, and
\[
M(x_{i+1})=y_i\qquad(0\le i<k).
\]
Its length is $2k+1$.

\begin{lemma}[Finite-stage augmentation]\label{lem:borel-augmentation}
Let $H$ be a finite-chart bipartite multigraph with left vertex space $L$ and right vertex space $R$, and let $M_0$ be a Borel matching in $H$. For every $q\ge1$ there is a Borel matching $M_q$ with no augmenting path of length at most $2q-1$. Every vertex matched by $M_0$ can be kept matched throughout the construction.
\end{lemma}

\begin{proof}
Suppose that the current matching $M$ has no augmenting path of length smaller than
\[
\ell=2k+1.
\]
We remove all augmenting paths of length $\ell$ in one finite Borel stage.

Orient every edge of $H$ outside $M$ from $L$ to $R$ and every edge of $M$ from $R$ to $L$, and denote the resulting directed graph by $\vec H_M$. Let
\[
d_M(v)=\dist_{\vec H_M}(L\setminus\dom M,v)
\]
be the directed distance from the unmatched left vertices. If
\[
P=(v_0,v_1,\ldots,v_\ell)
\]
is an $M$-augmenting path of length $\ell$, then
\begin{equation}\label{eq:borel-aug-distance}
d_M(v_j)=j
\qquad(0\le j\le\ell).
\end{equation}
Indeed, the path itself gives $d_M(v_j)\le j$. If $d_M(v_j)<j$, a directed path of length $<j$ from an unmatched left vertex to $v_j$, followed by the tail $(v_j,\ldots,v_\ell)$, gives a directed walk from an unmatched left vertex to an unmatched right vertex of length $<\ell$. A shortest such walk is an augmenting path, contrary to the assumption on $M$.

Recall that the finite-chart structure of $H$ is given by the finitely many Borel partial bijections
\[
\phi_1,\ldots,\phi_d.
\]
An augmenting path of length $\ell=2k+1$ is determined by its initial vertex together with the index sequence
\begin{equation}\label{eq:borel-aug-string}
(i_0,\ldots,i_k),
\qquad i_j\in\{1,\ldots,d\}.
\end{equation}
where $i_j$ is the index of the chart used by the $j$-th edge of the path directed from $L$ to $R$. Indeed, once the initial vertex and this index sequence are fixed, the intervening edges directed from $R$ to $L$ are determined by $M$. There are only finitely many index sequences \eqref{eq:borel-aug-string}. For a fixed index sequence, the set of initial vertices for which the resulting composition is defined and gives an augmenting path is Borel.

Two augmenting paths with the same index sequence \eqref{eq:borel-aug-string} and different initial vertices are vertex-disjoint. Indeed, if they meet in the same position, injectivity of the partial bijections $\phi_i$ and of $M^{-1}$ forces their initial vertices to be equal. If they meet in different positions, \eqref{eq:borel-aug-distance} gives two different values for $d_M$ at the same vertex.

Order the finitely many index sequences \eqref{eq:borel-aug-string}. For the first index sequence keep all its augmenting paths. For each later index sequence keep exactly those paths which avoid all paths already kept. The resulting Borel family $\mathcal Q$ is a maximal pairwise vertex-disjoint family of $M$-augmenting paths of length $\ell$. Indeed, every $M$-augmenting path of length $\ell$ which is not kept meets a path retained at an earlier or equal stage.

Flip $M$ simultaneously along all paths in $\mathcal Q$, and call the new matching $M'$. The two endpoints of every selected path become matched, and every internal vertex remains matched. Thus no vertex matched by $M$ becomes unmatched.

\medskip
\noindent
\textbf{Sublemma.}
The matching $M'$ has no augmenting path of length at most $\ell$.

\smallskip
\noindent
\emph{Proof.}
Suppose
\[
Q=(w_0,\ldots,w_m),\qquad m\le\ell,
\]
were such a path. Its endpoints are unmatched for $M'$, hence they were also unmatched for $M$. In particular,
\[
d_M(w_0)=0,
\qquad
d_M(w_m)\ge\ell,
\]
Indeed, if $d_M(w_m)<\ell$, there would be an $M$-directed path of length $<\ell$ from an unmatched left vertex to the unmatched right vertex $w_m$. Such a path is itself an $M$-augmenting path, contradicting the assumption on $M$.

Along each directed edge of $Q$, oriented according to $M'$, the old distance $d_M$ increases by at most one. Indeed, if the matching status of that edge was not changed by the flip, then its $M'$-orientation agrees with its $M$-orientation, so the definition of directed distance in $\vec H_M$ gives an increase of at most one. If the matching status of that edge was changed by the flip, then its $M'$-orientation is opposite to the direction in which the selected augmenting path traversed it, and by \eqref{eq:borel-aug-distance} the old distance $d_M$ decreases by one. Hence the endpoint distance bounds above imply
\begin{equation}\label{eq:borel-aug-newdist}
m=\ell,
\qquad
d_M(w_j)=j\quad(0\le j\le\ell).
\end{equation}
and $Q$ uses no edge whose status was changed.

Moreover, $Q$ cannot meet any of the selected augmenting paths. Indeed, every vertex of a selected path is matched by $M'$. If $Q$ passed through such a vertex, then, since it is an internal vertex of the augmenting path $Q$, $Q$ would have to use the unique $M'$-matching edge incident to it. That edge belongs to the selected path and has changed its matching status, contradicting \eqref{eq:borel-aug-newdist}. Thus $Q$ is disjoint from $\mathcal Q$.

Since $Q$ uses no changed edge, it was already an $M$-augmenting path of length $\ell$. This contradicts maximality of $\mathcal Q$. \hfill$\square$

\medskip
Thus all length-$\ell$ augmenting paths have been removed without creating any augmenting path of length at most $\ell$.

Starting with $M_0$, perform this construction for
\[
\ell=1,3,\ldots,2q-1.
\]
Only finitely many Borel operations are used, and previously matched vertices remain matched at every stage.
\end{proof}

\subsection{Uniform F\o lner comparison}

\begin{theorem}[Uniform F\o lner comparison]\label{thm:borel-comparison}
Let $A,B\subseteq X$ be Borel. Suppose that there are $\rho,\gamma>0$ such that every nonempty finite $\rho$-F\o lner set $F$ satisfies
\begin{equation}
|B\cap F|-|A\cap F|\ge\gamma |F|.
\label{eq:borel-surplus}
\end{equation}
If $G$ is uniformly Borel amenable and has the F\o lner property, then
\[
A\preceqG B.
\]
\end{theorem}

\begin{proof}
The proof has four stages. First we use the uniform F\o lner surplus and Packing input I to construct a bounded-displacement injection between finite amplifications of $A$ and $B$. Second we show that, for every partial matching in the resulting finite-chart bipartite multigraph, the unmatched set on the $B$-side is uniformly dense. Third we choose a matching with no short augmenting paths and use the associated alternating map to produce many disjoint copies of the unmatched set on the $A$-side. Finally we combine the two estimates and repeatedly apply the three-to-two lemma.

\medskip
\noindent
\textit{From the F\o lner surplus to a bounded many-to-many injection.}

Choose an integer $\ell\ge1$ with
\begin{equation}
\ell\gamma\ge1.
\label{eq:borel-ell-choice}
\end{equation}
If $F$ is $\rho$-F\o lner, then
\[
\begin{aligned}
\ell|B\cap F|-(\ell+1)|A\cap F|
&=\ell\bigl(|B\cap F|-|A\cap F|\bigr)-|A\cap F|\\
&\ge \ell\gamma |F|-|A\cap F|\\
&\ge0.
\end{aligned}
\]
Thus
\begin{equation}
(\ell+1)|A\cap F|\le \ell|B\cap F|.
\label{eq:borel-tile-surplus}
\end{equation}

Choose
\begin{equation}
0<\theta<\min\left\{\rho,\frac1{\ell+1}\right\}
\label{eq:borel-theta-choice}
\end{equation}
and apply Packing input I. Let $\mathcal P_1,\ldots,\mathcal P_m$ be the resulting packings. Since every tile is $\theta$-F\o lner, it is $\rho$-F\o lner, and \eqref{eq:borel-tile-surplus} applies to every tile. Hence, for $P\in\mathcal P_i$, there is an injection
\begin{equation}
\eta_{i,P}:(A\cap P)\times\{1,\ldots,\ell+1\}
\longrightarrow
(B\cap P)\times\{1,\ldots,\ell\}.
\label{eq:borel-tile-injection}
\end{equation}
The injections in \eqref{eq:borel-tile-injection} may be chosen in a Borel way. Fix a Borel order on $X$. On each tile $P$, list $A\cap P$ and $B\cap P$ in this order and take the lexicographically first injection satisfying \eqref{eq:borel-tile-injection}. Since the tiles have uniformly bounded cardinality, this rule is Borel.

Let
\[
S=\{(a,i,r):a\in A\cap[\mathcal P_i],\ 1\le r\le\ell+1\}
\]
and
\[
R=\{(b,i,s):b\in B\cap[\mathcal P_i],\ 1\le s\le\ell\}.
\]
The maps \eqref{eq:borel-tile-injection} combine to a Borel injection $\eta:S\to R$. Put
\begin{equation}
N=\ell m.
\label{eq:borel-amplification-size}
\end{equation}
For every $a\in A$, \eqref{eq:borel-theta-choice} and Packing input I give
\begin{equation}
\bigl|\{(i,r):(a,i,r)\in S\}\bigr|
>(\ell+1)(1-\theta)m
>\ell m=N.
\label{eq:borel-source-label-count}
\end{equation}
Choose the first $N$ pairs $(i,r)$ in a fixed order and identify them with $1,\ldots,N$. Also identify the $\ell m=N$ target pairs $(i,s)$ with $1,\ldots,N$. Restricting $\eta$ to the chosen source pairs gives a Borel injection
\begin{equation}
J:A\times\{1,\ldots,N\}
\longrightarrow
B\times\{1,\ldots,N\}.
\label{eq:borel-many-to-many}
\end{equation}
There is $D<\infty$ such that
\begin{equation}
J(a,k)=(b,j)\quad\Longrightarrow\quad d_G(a,b)\le D,
\label{eq:borel-many-to-many-displacement}
\end{equation}
because every map in \eqref{eq:borel-tile-injection} stays inside one tile and the tiles have a common diameter bound.

For $1\le k,j\le N$, put
\[
D_{k,j}=\{a\in A:J(a,k)\in B\times\{j\}\}
\]
and define
\begin{equation}
\phi_{k,j}:D_{k,j}\longrightarrow B,
\qquad
\phi_{k,j}(a)=b
\quad\Longleftrightarrow\quad
J(a,k)=(b,j).
\label{eq:borel-matching-charts}
\end{equation}
Since $J$ is injective, each $\phi_{k,j}$ is a Borel partial injection. We regard the finitely many maps \eqref{eq:borel-matching-charts} as the edge charts of a bipartite multigraph $\mathcal H$ with left side $A$ and right side $B$. Parallel edges are allowed: different pairs $(k,j)$ may join the same two vertices.

\medskip
\noindent
\textit{Uniform density of the unmatched $B$-set.}

We first show that for \emph{every} partial matching $M$ in $\mathcal H$, the unmatched set on the $B$-side is uniformly dense, with constants independent of $M$.

Let $M$ be any partial matching in $\mathcal H$, and put
\begin{equation}
V=B\setminus\ran M.
\label{eq:borel-unmatched-right}
\end{equation}
Let $F\subseteq X$ be an arbitrary nonempty finite $\rho$-F\o lner set. Define
\[
N_D(F)=\{x\in X:d_G(x,F)\le D\}
\]
and
\begin{equation}
c_D=\sup_{x\in X}|B_D(x)|<\infty.
\label{eq:borel-ball-size-D}
\end{equation}
Every $b\in\ran M\cap F$ is matched to a distinct point of $A\cap N_D(F)$. Therefore
\begin{equation}
|\ran M\cap F|
\le |A\cap F|+|N_D(F)\setminus F|.
\label{eq:borel-matched-count}
\end{equation}
Every point $y\in N_D(F)\setminus F$ lies in $B_D(x)$ for some $x\in\partial F$: take a path of length at most $D$ from $F$ to $y$ and let $x$ be its last vertex in $F$. Hence
\begin{equation}
|N_D(F)\setminus F|\le c_D|\partial F|.
\label{eq:borel-collar-count}
\end{equation}
If $F$ is $\rho$-F\o lner, \eqref{eq:borel-surplus}, \eqref{eq:borel-unmatched-right}, \eqref{eq:borel-matched-count}, and \eqref{eq:borel-collar-count} give
\begin{equation}
|V\cap F|
\ge \gamma |F|-c_D|\partial F|.
\label{eq:borel-right-surplus}
\end{equation}
Choose $0<\delta\le\rho$ so that
\begin{equation}
c_D\delta<\frac\gamma2.
\label{eq:borel-delta-choice}
\end{equation}
Then every nonempty finite $\delta$-F\o lner set meets $V$. By the F\o lner property there is $r<\infty$ such that every $B_r(x)$ contains such a set. Hence
\begin{equation}
d_G(x,V)\le r
\qquad(x\in X).
\label{eq:borel-right-density}
\end{equation}
The number $r$ depends only on \eqref{eq:borel-surplus} and the displacement bound $D$, and not on the matching $M$.

Fix a Borel linear order $<$ on $X$ and define
\begin{equation}
\varphi(x)=\min_{<}\{v\in V:d_G(x,v)=d_G(x,V)\}.
\label{eq:borel-nearest-unmatched}
\end{equation}
The map $\varphi$ is Borel: by \eqref{eq:borel-right-density}, one only has to inspect the finite ball $B_r(x)$, first finding the least distance at which it meets $V$ and then taking the $<$-least point at that distance. Thus $d_G(x,\varphi(x))\le r$. Put
\begin{equation}
L=\sup_{x\in X}|B_r(x)|.
\label{eq:borel-fibre-bound}
\end{equation}
Every fibre of $\varphi$ has cardinality at most $L$. If
\[
\iota(x)=1+\bigl|\{y:\varphi(y)=\varphi(x),\ y<x\}\bigr|,
\]
then
\[
x\longmapsto (\varphi(x),\iota(x))
\]
is a Borel injection, and therefore
\begin{equation}
X\preceqG LV.
\label{eq:borel-space-into-right}
\end{equation}

\medskip
\noindent
\textit{A matching with no short augmenting paths.}

We now choose a particular matching and exploit the absence of short augmenting paths. The goal of this stage is to show that, if $U\subseteq A$ is the unmatched set on the left and $V\subseteq B$ is the unmatched set on the right, then sufficiently many labelled copies of $U$ can be injected back into labelled copies of $A$.

Fix $q\ge1$ for the moment. By Lemma~\ref{lem:borel-augmentation}, starting from the empty matching, after finitely many Borel augmentation stages we obtain a Borel matching $M$ in $\mathcal H$ with no augmenting path of length at most $2q-1$. Put
\begin{equation}
U=A\setminus\dom M,
\qquad
V=B\setminus\ran M.
\label{eq:borel-unmatched-sides}
\end{equation}
The estimate \eqref{eq:borel-right-density}, and hence \eqref{eq:borel-space-into-right}, applies to this $V$.

Define the partial map
\begin{equation}
T=(M^{-1}\times\id_{\{1,\ldots,N\}})\circ J
\label{eq:borel-alternating-map}
\end{equation}
on $A\times\{1,\ldots,N\}$. Thus if
\[
J(a,i)=(b,j)
\]
and $b$ is matched by $M$ to $a'$, then
\begin{equation}
T(a,i)=(a',j).
\label{eq:borel-alternating-map-value}
\end{equation}
Both $J$ and $M^{-1}\times\id$ are injective, so $T$ is injective on its domain.

\medskip
Here $T^0$ denotes the identity map on $A\times[N]$, and $\pi_A(a,i)=a$.

\noindent
\textbf{Sublemma.}
For every $0\le k\le q$, the iterate $T^k$ is defined on all of $U\times[N]$. Moreover, the sets
\[
T^0(U\times[N]),\ T^1(U\times[N]),\ \ldots,\ T^q(U\times[N])
\]
are pairwise disjoint, and for every $(u,i)\in U\times[N]$,
\begin{equation}
d_G\bigl(u,\pi_A(T^k(u,i))\bigr)\le 2kD.
\label{eq:borel-iterate-displacement}
\end{equation}

\smallskip
\noindent
\emph{Proof.}
We first prove that all the iterates are defined. Suppose otherwise, and let $k\le q$ be the first index for which an application of $T$ fails. Write
\[
T^{r}(u,i_0)=(a_r,i_r),
\qquad 0\le r\le k-1.
\]
For $1\le r\le k-1$, there is a vertex $b_r\in B$ such that
\begin{equation}
J(a_{r-1},i_{r-1})=(b_r,i_r),
\qquad
M(a_r)=b_r.
\label{eq:borel-alternating-data}
\end{equation}
At the failed step,
\[
J(a_{k-1},i_{k-1})=(b_k,i_k)
\]
for an unmatched vertex $b_k\in B$. Hence
\begin{equation}
(a_0=u),b_1,a_1,b_2,\ldots,a_{k-1},b_k
\label{eq:borel-alternating-walk}
\end{equation}
is a walk in $\mathcal H$. For each $r<k$, the edge from $b_r$ to $a_r$ belongs to $M$. If the preceding $J$-edge from $a_{r-1}$ to $b_r$ also belongs to $M$, then the matching property forces it to be the same edge, so the two consecutive steps form an immediate backtrack $a_{r-1},b_r,a_{r-1}$. Delete every such two-edge backtrack. In the remaining walk every left-to-right edge lies outside $M$ and every right-to-left edge lies in $M$. The last $J$-edge cannot belong to $M$, since $b_k$ is unmatched. We therefore obtain a directed alternating walk from the unmatched left vertex $u$ to the unmatched right vertex $b_k$, of length at most
\[
2k-1\le2q-1.
\]
A shortest directed walk with these endpoints is an augmenting path, contradicting the choice of $M$. Hence all the iterates are defined.

They are pairwise disjoint. If $0\le i<j\le q$ and
\[
T^i(u,s)=T^j(u',s'),
\]
then injectivity of $T^i$ gives
\[
(u,s)=T^{j-i}(u',s').
\]
The right-hand side belongs to $\ran T\subseteq\dom M\times[N]$, whereas the left-hand side belongs to $U\times[N]$, a contradiction.

Finally, each edge of $\mathcal H$ has $G$-displacement at most $D$ by \eqref{eq:borel-many-to-many-displacement}. One application of $T$ traverses one $J$-edge and one matching edge backwards, so the underlying left vertex after one application is at $G$-distance at most $2D$ from the preceding one. After $k$ iterations the displacement is at most $2kD$, proving \eqref{eq:borel-iterate-displacement}. \hfill$\square$

\medskip
The sublemma shows that the disjoint maps $T^0,\ldots,T^q$ give a bounded-displacement Borel injection
\begin{equation}
(q+1)NU\preceqG NA.
\label{eq:borel-disjoint-iterates}
\end{equation}

\medskip
\noindent
\textit{Compression by the three-to-two lemma.}

We now have the two estimates needed to compare the unmatched sets: \eqref{eq:borel-space-into-right} says that the whole space can be injected into finitely many copies of $V$, while \eqref{eq:borel-disjoint-iterates} says that sufficiently many copies of $U$ fit inside finitely many copies of $A$. We choose the amplification constants so that repeated three-to-two compression turns these two estimates into a single injection $U\to V$.

Choose $h$ so that
\begin{equation}
2^h\ge NL,
\label{eq:borel-h-choice}
\end{equation}
and then choose $q$ so that
\begin{equation}
(q+1)N\ge3^h.
\label{eq:borel-q-choice}
\end{equation}
Take the matching $M$ above for this value of $q$. From \eqref{eq:borel-space-into-right}, \eqref{eq:borel-disjoint-iterates}, \eqref{eq:borel-h-choice}, and \eqref{eq:borel-q-choice},
\begin{equation}
3^hU
\preceqG (q+1)NU
\preceqG NA
\preceqG NX
\preceqG NLV
\preceqG 2^hV.
\label{eq:borel-compression-chain}
\end{equation}
Finite amplification does not affect uniform Borel amenability, so the Borel three-to-two lemma applies without change to finite labelled copies of $X$.

Starting from the first and last terms of \eqref{eq:borel-compression-chain},
\[
3^hU\preceqG 2^hV,
\]
apply the Borel three-to-two lemma with
\[
A_1=3^{h-1}U,
\qquad
B_1=2^{h-1}V.
\]
Since $3A_1=3^hU$ and $2B_1=2^hV$, we obtain
\[
3^{h-1}U\preceqG 2^{h-1}V.
\]
Apply the same lemma again with
\[
A_2=3^{h-2}U,
\qquad
B_2=2^{h-2}V,
\]
and continue in the same way. After $h$ applications we have
\begin{equation}
3^hU\preceqG2^hV
\;\Longrightarrow\;
3^{h-1}U\preceqG2^{h-1}V
\;\Longrightarrow\;\cdots\;\Longrightarrow\;
3U\preceqG2V
\;\Longrightarrow\;
U\preceqG V.
\label{eq:borel-compression-iteration}
\end{equation}
The matching $M$ gives an injection
\[
A\setminus U\longrightarrow B\setminus V,
\]
and \eqref{eq:borel-compression-iteration} gives an injection $U\to V$. Their union is a bounded-displacement Borel injection $A\to B$.
\end{proof}

\subsection{Completion of the tiling}

\begin{proof}[Proof of Theorem~\ref{thm:borel-main}]
It is enough to consider $0<\varepsilon<1$. Put
\begin{equation}
t=\frac\varepsilon{10},
\qquad
N_0=\left\lceil\frac1t\right\rceil.
\label{eq:borel-absorption-parameters}
\end{equation}
Every connected component with at most $N_0$ vertices is taken as one part of the final partition. Remove these components and work on the invariant Borel set consisting of the remaining components.

Apply Packing input II with parameter $t$. Let $\mathcal T$ be the resulting packing and let $\rho>0$ be given by \eqref{eq:borel-packing-II}. Put
\begin{equation}
A=X\setminus[\mathcal T].
\label{eq:borel-uncovered}
\end{equation}
Every $T\in\mathcal T$ satisfies
\begin{equation}
|T|>\frac1t.
\label{eq:borel-large-tiles}
\end{equation}
Indeed, if $\partial T\ne\varnothing$, then
\[
1\le|\partial T|<t|T|.
\]
If $\partial T=\varnothing$, then $T$ is a whole connected component, and all components of size at most $N_0$ have already been removed.

Since $t<1/10$, \eqref{eq:borel-large-tiles} gives
\begin{equation}
\lceil3t|T|\rceil\le3t|T|+1<4t|T|<|T|.
\label{eq:borel-reserve-rounding}
\end{equation}
Thus every tile contains at least $\lceil3t|T|\rceil$ points. Let $B_T$ be the first
\[
\lceil 3t|T|\rceil
\]
points of $T$ in a fixed Borel order, and put
\begin{equation}
B=\bigcup_{T\in\mathcal T}B_T.
\label{eq:borel-reserve}
\end{equation}
By \eqref{eq:borel-large-tiles},
\begin{equation}
3t|T|\le |B_T|\le 3t|T|+1\le4t|T|.
\label{eq:borel-reserve-bounds}
\end{equation}

Let $F$ be a nonempty finite $\rho$-F\o lner set. By \eqref{eq:borel-packing-II},
\begin{equation}
\left|\bigcup\{T\in\mathcal T:T\subseteq F\}\right|\ge(1-t)|F|.
\label{eq:borel-contained-tile-mass}
\end{equation}
Since every point of $A\cap F$ is outside the union on the left,
\begin{equation}
|A\cap F|\le t|F|.
\label{eq:borel-uncovered-count}
\end{equation}
Using the lower bound in \eqref{eq:borel-reserve-bounds} on every tile contained in $F$ gives
\begin{equation}
|B\cap F|\ge 3t(1-t)|F|.
\label{eq:borel-reserve-count}
\end{equation}
Therefore
\begin{equation}
|B\cap F|-|A\cap F|
\ge (2t-3t^2)|F|.
\label{eq:borel-absorption-surplus}
\end{equation}
By the Uniform F\o lner comparison theorem, there is a bounded-displacement Borel injection
\begin{equation}
\psi:A\longrightarrow B.
\label{eq:borel-absorption-injection}
\end{equation}
Let $s$ be a displacement bound for $\psi$.

For $T\in\mathcal T$, define
\begin{equation}
T'=T\cup\psi^{-1}(B_T).
\label{eq:borel-enlarged-tiles}
\end{equation}
Since the sets $B_T$ partition $B$ and $\psi$ is injective, the sets $T'$ form a Borel partition of the remaining space. By \eqref{eq:borel-reserve-bounds},
\begin{equation}
|T'\setminus T|\le |B_T|\le4t|T|.
\label{eq:borel-enlargement-size}
\end{equation}
Moreover,
\[
\partial T'\subseteq \partial T\cup(T'\setminus T),
\]
so
\begin{equation}
|\partial T'|
<t|T|+4t|T|
=5t|T|
\le5t|T'|
<\varepsilon|T'|.
\label{eq:borel-enlargement-boundary}
\end{equation}
If the tiles in $\mathcal T$ have diameter at most $k$, every point of $T'\setminus T$ is within $G$-distance $s$ of a point of $B_T\subseteq T$. Hence
\begin{equation}
\operatorname{diam}_G(T')\le k+2s.
\label{eq:borel-enlargement-diameter}
\end{equation}
Reinserting the components removed at the beginning gives the required partition of $X$.
\end{proof}

\section{The clopen case}\label{sec:clopen}

In this section we work with the stable orbit structure $\Gamma\curvearrowright X$ fixed in Section~2, and we use throughout the terminology and notation of the Preliminaries. In particular, clopen orbit maps, bounded-degree clopen graphs, finite amplifications, clopen packings and tilings, invariant measures $\M$, topological amenability, the relation $A\preceqcl B$, the inner boundary $\partial_KF$, and the notions of almost finiteness and strong almost finiteness are all understood in the sense of Section~2. We shall use repeatedly the tower description of clopen packings, including the terminology of bases and levels introduced there. We shall state separately when minimality or the existence of an invariant probability measure is needed. As in the Preliminaries, the points of each tile are obtained from a common base point by clopen orbit maps and therefore lie in one full $\Gamma$-orbit, while their boundary is always measured in the specified bounded-degree clopen graph.

For later use, if $H$ is a bounded-degree clopen graph and $s\geq1$, we write $H^{[s]}$ for the graph joining two distinct points whenever their $H$-distance is at most $s$.

We shall also use the following elementary consequence of the clopen graph structure. A \emph{finite clopen proper colouring} of a bounded-degree clopen graph is a finite clopen partition of its vertex space such that no part contains both endpoints of an edge. This is the clopen counterpart of the finite Borel colouring theorem of Kechris--Solecki--Todorcevic \cite{KST}. Such a colouring always exists. Indeed, the edge relation is compact and disjoint from the diagonal, so every point has a clopen neighbourhood containing no edge. Compactness gives a finite cover by such sets, and a disjoint clopen refinement gives the required colouring.

We record one consequence of invariance that will be used repeatedly. Let $\mathcal E$ be a clopen tiling and let $A\subseteq X$ be clopen. Then, for every $\mu\in\M$,
\begin{equation}\label{eq:tower-average}
 \mu(A)=\int_X\frac{|A\cap\mathcal E(x)|}{|\mathcal E(x)|}\,d\mu(x).
\end{equation}
To see this, work tower by tower. All levels of one tower have the same measure, since the level maps are clopen orbit maps. Averaging the indicator of $A$ over the points of each tile and then integrating over the tower therefore gives exactly the measure of $A$ on that tower. Summing over the finitely many towers proves the identity. In particular,
\begin{equation}\label{eq:boundary-average}
 \mu(D_K(\mathcal E))
 =\int_X\frac{|\partial_K\mathcal E(x)|}{|\mathcal E(x)|}\,d\mu(x).
\end{equation}
More generally, if clopen sets $A,B\subseteq X$ satisfy $|A\cap T|\leq|B\cap T|$ in every tile $T$ of one clopen tiling, then $\mu(A)\leq\mu(B)$ for every $\mu\in\M$.

Write $\1_A$ for the indicator of a set $A$. The main result of this section is the following.

\begin{theorem}[The clopen case]\label{thm:clopen-main}
Suppose that $\Gamma\curvearrowright X$ is minimal, that its orbit relation is topologically amenable, and that $\M\ne\varnothing$. For every bounded-degree clopen graph $K$ and every $\tau,\eta>0$, there are clopen tilings $\mathcal E_1,\ldots,\mathcal E_m$ such that every tile of every $\mathcal E_i$ is $\tau$-F\o lner for $K$ and
\begin{equation}\label{eq:strong-conclusion}
 \frac1m\sum_{i=1}^m\1_{D_K(\mathcal E_i)}(x)<\eta
 \qquad(x\in X).
\end{equation}
Thus the orbit relation is strongly almost finite, and in particular almost finite.
\end{theorem}

The point of the formulation is that minimality refers to the full $\Gamma$-orbits; no connected component of a prescribed graph $K$ is assumed to be dense or to have a uniform F\o lner property in its own graph metric. The proof first constructs a clopen F\o lner packing whose union has measure close to one for every $\mu\in\M$. We then find a clopen orbit injection from its uncovered complement into a small reserved subset of the packing, and use this injection to enlarge the tiles to a tiling of $X$. The construction of finite families of tilings follows the method of Elek--Tim\'ar \cite[Lemma~2.2]{ElekTimarUBA}; the selection of one large packing uses their local replacement method \cite[Proposition~3.1]{ElekTimarUBA}. The required injection is obtained using the implication $3A\preceqcl2B\Rightarrow A\preceqcl B$, proved below following Glasner--Liu \cite[Sections~3--4]{GlasnerLiu}.

\subsection{Finite clopen partitions from amenability}

Let $H$ be a topologically amenable bounded-degree clopen graph and let $\alpha>0$. By the formulation of topological amenability in the Preliminaries, there are $R<\infty$, an integer $n\geq1$, and a continuous family of probabilities $p_x$ on the $H$-component of $x$ such that, for every $x\in X$,
\begin{equation}\label{eq:localized-kernel}
 \supp p_x\subseteq B_H(x,R),
 \qquad np_x(z)\in\{0,1,\ldots,n\}\quad(z\in X),
 \qquad \|p_x-p_y\|_1<\alpha\quad(x\sim_Hy).
\end{equation}
The common-denominator rational form is obtained by starting with a smaller error and approximating the finitely many continuous weights on a finite clopen partition by rational probability vectors with a common denominator. The support remains inside the same radius-$R$ $H$-ball. This is the localized amenability property recorded in the Preliminaries, based on \cite[Proposition~5.1.1]{ADR}; see also \cite[Lemma~3.3]{AustadBonicke}.

The Borel counterpart of the following proposition is \cite[Lemma~2.2]{ElekTimarUBA}.

\begin{proposition}[Finite random clopen tilings]\label{prop:random-tilings}
Let $H$ be a topologically amenable bounded-degree clopen graph. For every $\beta>0$, there are full clopen tilings $\mathcal E_1,\ldots,\mathcal E_q$ such that, for every $x\in X$,
\begin{equation}\label{eq:random-interior}
 \frac1q\sum_{j=1}^{q}\1_{D_H(\mathcal E_j)}(x)<\beta.
\end{equation}
The tilings have a common finite orbit bound in the sense of \eqref{eq:orbit-bounded-tiles} of the Preliminaries. No F\o lner condition on their tiles is asserted.
\end{proposition}

\begin{proof}
Let $d\geq1$ bound the degree of $H$, and choose \eqref{eq:localized-kernel} with $\alpha<\beta/d$. Since $H$ is given by finitely many clopen orbit maps and $R$ is fixed, there is a finite set $D\subseteq\Gamma$ such that, for every $x\in X$,
\[
 B_H(x,R)\subseteq Dx.
\]
For each $x$, replace every support point $z$ of weight $m/n$ by $m$ labelled copies. Thus $p_x$ is represented by the $n$-element set
\[
 Q_x=\{(z,k):z\in\supp p_x,\ 1\leq k\leq np_x(z)\}.
\]
We call the elements of $Q_x$ tokens. Since $p_x$ is a probability measure, $Q_x$ contains $n$ tokens.

Choose a finite clopen colouring $\chi$ such that, whenever $x=y$ or $x\sim_Hy$, all distinct points occurring in $\supp p_x\cup\supp p_y$ have distinct colours. To obtain it, join distinct points $z=gx$ and $z'=hy$ whenever $g,h\in D$ and $x=y$ or $x\sim_Hy$. This is a bounded-degree clopen graph: its pairs are described by finitely many compositions of the maps in $D$ and the charts of $H$. Apply the finite clopen colouring observation above. Give $(z,k)$ the colour $(\chi(z),k)$. The tokens of $Q_x\cup Q_y$ now have distinct colours for adjacent $x,y$, and the same is true within each $Q_x$.

\begin{lemma}[Clopen tiling associated with an ordering]\label{lem:ordering-tiling}
For every ordering $\omega$ of the finite set of token colours, let $f_\omega(x)$ be the token of $Q_x$ whose colour is highest in $\omega$, and define
\[
 x\,E_\omega\,y\quad\Longleftrightarrow\quad f_\omega(x)=f_\omega(y).
\]
Then the equivalence classes of $E_\omega$ form a full clopen tiling $\mathcal E_\omega$. The finite set $D^{-1}D$ is a common orbit bound for all these tilings.
\end{lemma}

\begin{proof}
Fix $\omega$. For $g\in D$ and $1\leq k\leq n$, the set of points $x$ for which $f_\omega(x)=(gx,k)$ is clopen. Indeed, the condition $(gx,k)\in Q_x$ is $np_x(gx)\geq k$; the function $x\mapsto p_x(gx)$ is continuous and takes values in the finite set $\{0,1/n,\ldots,1\}$. Whether $(gx,k)$ has higher colour than every other token in $Q_x$ is decided by finitely many conditions of the same form together with the clopen colouring $\chi$. Thus $X$ has a finite clopen partition on each part of which $f_\omega$ has the form $x\mapsto(gx,k)$.

On two such parts, where $f_\omega(x)=(gx,k)$ and $f_\omega(y)=(hy,\ell)$, equality is possible only when $k=\ell$, and then it is equivalent to $gx=hy$, or $y=h^{-1}gx$. Hence $E_\omega$ is a finite union of graphs of clopen orbit maps and is clopen. Every $E_\omega$-class is finite: if $f_\omega(x)=(z,k)$, then $x=g^{-1}z$ for some $g\in D$, so the class has at most $|D|$ points. Since every $x\in X$ has a unique highest-ranked token, these finite classes cover $X$ and therefore form a full clopen tiling $\mathcal E_\omega$. Finally, if $x\,E_\omega\,y$, then $y=h^{-1}gx$ for some $g,h\in D$, so $D^{-1}D$ is a common orbit bound.
\end{proof}

Fix an $H$-edge $xy$, and choose $\omega$ uniformly from all orderings. Every token of $Q_x\cup Q_y$ is equally likely to have the highest colour. The two choices differ exactly when that highest token lies in the symmetric difference. Hence
\begin{equation}\label{eq:token-disagreement}
 \Pr\{f_\omega(x)\ne f_\omega(y)\}
 =\frac{|Q_x\mathbin\triangle Q_y|}{|Q_x\cup Q_y|}
 \leq\frac{n\|p_x-p_y\|_1}{n}<\alpha.
\end{equation}
The point $x$ is a boundary point of its fibre only if this disagreement occurs for one of its at most $d$ neighbours. The union bound gives a probability smaller than $d\alpha<\beta$. Listing all the orderings proves \eqref{eq:random-interior}.
\end{proof}

\begin{corollary}[Separated clopen packings]\label{cor:separated-packings}
Let $H$ be a topologically amenable bounded-degree clopen graph. For every $s\geq1$ and $\theta>0$, there are finitely many clopen packings $\mathcal P_1,\ldots,\mathcal P_m$ such that distinct tiles of each packing have $H$-distance greater than $s$ and, for every $x\in X$,
\begin{equation}\label{eq:separated-coverage}
 |\{i:x\in[\mathcal P_i]\}|>(1-\theta)m.
\end{equation}
Their tiles have a common finite orbit bound, but are not required to be F\o lner.
\end{corollary}

\begin{proof}
The graph $H^{[s]}$ is topologically amenable by the Preliminaries. Apply Proposition~\ref{prop:random-tilings} to $H^{[s]}$ with error $\theta$. From each tile retain only the points whose entire radius-$s$ $H$-ball lies in that tile, and discard empty sets. These are clopen packings, and \eqref{eq:random-interior} gives \eqref{eq:separated-coverage}. Two retained points in different tiles cannot be at $H$-distance at most $s$, because then the ball of the first would contain the second. The finite orbit bound is inherited from the tilings.
\end{proof}

\subsection{The clopen three-to-two principle}

We retain the three-to-two argument, including the finite matching construction. Its finite-order reduction and short-augmenting-path method follow Glasner--Liu \cite[Sections~3--4]{GlasnerLiu}. In the tree part below we apply the localized amenability property \eqref{eq:localized-kernel} to an auxiliary bounded-degree clopen graph on a finite amplification.
\begin{proposition}[Clopen three-to-two]\label{prop:clopen-three-two}
Suppose that the orbit relation of $\Gamma\curvearrowright X$ is topologically amenable. If $A,B\subseteq X$ are clopen and $3A\preceqcl2B$, then $A\preceqcl B$. More precisely, given a clopen injection $J:A\times[3]\to B\times[2]$, there is a clopen injection $I:A\to B$ such that $I(a)$ is the underlying $B$-coordinate of one of $J(a,1),J(a,2),J(a,3)$ for every $a\in A$. The same assertion holds in finite amplifications; minimality and an invariant measure are not required.
\end{proposition}

\begin{proof}
If $A=\varnothing$, there is nothing to prove.  Fix the injection $J$ from the statement.  We first introduce the finite clopen objects that will be used throughout the proof.

\medskip
\noindent\textbf{Setup.}
Put
\[
   Y=B\times[2],\qquad
   J_i(a)=J(a,i),\qquad
   F_i=J_i(A),\qquad
   F=F_1\sqcup F_2\sqcup F_3.
\]
Define a clopen homeomorphism $p$ of $Y$ by cycling
\[
   J_1(a)\mapsto J_2(a)\mapsto J_3(a)\mapsto J_1(a)
\]
and fixing $Y\setminus F$.  Let $q$ interchange the two copies of $B$.  Thus
\[
      p^3=q^2=\id.
\]

Let $V_3=F_1$, let $V_1=Y\setminus F$, and put $V=V_3\sqcup V_1$.  Define
\[
  \pi:Y\longrightarrow V
\]
by $\pi=\id,p^{-1},p^{-2}$ on $F_1,F_2,F_3$, respectively, and by $\pi=\id$ on $V_1$.  Define the clopen multigraph $H$ with vertex set $V=V_3\sqcup V_1$ by giving one labelled edge $e_b$ to each $b\in B$, with endpoints
\begin{equation}
      \pi(b,1),\qquad \pi(b,2).
      \label{eq:clopen-incidence-edge}
\end{equation}  A loop is allowed, distinct $b$'s may give parallel edges, and a loop contributes two to the degree.

The maps $p$ and $q$ also generate an action of $\Pi=C_3*C_2$ on $Y$.  Its labelled Schreier graph has the three-point $p$-orbits in $F$ as triangles, while the $p$-orbits in $Y\setminus F$ are singletons; the $q$-edge indexed by $b$ joins $(b,1)$ to $(b,2)$.  If every $p$-orbit is collapsed to a single vertex and the labelled $q$-edges are retained, the resulting multigraph is exactly $H$.  Thus $H$ is the collapsed graph in which the three points associated with $a$ are represented by the single vertex $v_a$.  We use $H$ for the edge-selection problem; when amenability is invoked below, we return to the uncollapsed graph on $Y$.

Counting the three points in a $p$-orbit shows that every vertex of $V_3$ has degree three, while every vertex of $V_1$ has degree one.  The degree-three vertex corresponding to $a\in A$ is
\[
      v_a=J_1(a).
\]

The required injection is equivalent to a collision-free choice of one incident edge at every vertex of $V_3$.  Indeed, if $v_a$ chooses $e_b$, define $I(a)=b$.  By \eqref{eq:clopen-incidence-edge}, one of the three points $J(a,i)$ equals $(b,j)$ for some $j\in[2]$, so $b$ is one of the prescribed targets.  If two distinct vertices $v_a,v_{a'}$ had the same image $b$, then the same non-loop edge $e_b$ would have been chosen at both endpoints.  A loop has only one degree-three endpoint and causes no collision.

It will also be useful to keep the corresponding bipartite description.  For $i\in[3]$ and $j\in[2]$, let
\begin{equation}
 D_{ij}=\{a\in A:J(a,i)\in B\times\{j\}\},
 \qquad
 \phi_{ij}(a)=b\quad\text{if }J(a,i)=(b,j).
 \label{eq:clopen-six-charts}
\end{equation}
These six charts form a bounded-degree clopen bipartite multigraph from $A$ to $B$.  A clopen matching covering all of $A$ gives exactly the required injection.

Equivalently, a clopen set $C\subseteq F$ specifies a partial matching when
\begin{equation}
   C,\ pC,\ p^2C \text{ are pairwise disjoint},
   \qquad C\cap qC=\varnothing.
   \label{eq:clopen-partial-matching}
\end{equation}
It covers all left vertices exactly when
\begin{equation}
     F=C\sqcup pC\sqcup p^2C.
     \label{eq:clopen-full-matching}
\end{equation}
The selected edge at $a$ is the one whose point $J(a,i)$ belongs to $C$.

We now separate the genuinely tree-like part of $H$.  For $n\ge1$, let $W_n\subseteq V_3$ be the set of vertices whose radius-$n$ ball in $H$ contains either a vertex of $V_1$ or a cycle.  Here degree is always computed in $H$, not in the ball; a loop is a cycle of length one and a pair of parallel edges is a cycle of length two.  By the clopen local-neighbourhood property, every $W_n$ is clopen.  Hence
\begin{equation}
       Z_3=V_3\setminus\bigcup_{n\ge1}W_n
       \label{eq:clopen-tree-components}
\end{equation}
is closed.  It is exactly the union of the three-regular tree components of $H$, viewed in $V_3$.  Put
\begin{equation}
       Z=\pi^{-1}(Z_3)=Z_3\sqcup pZ_3\sqcup p^2Z_3.
       \label{eq:clopen-uncollapsed-tree}
\end{equation}
Then $Z$ is a closed subset of $Y$ and is invariant under the action of $\Pi=C_3*C_2$.  Indeed, $p$-invariance is immediate.  For $q$, observe that a $q$-edge projects under $\pi$ to an edge of $H$; an endpoint of an edge lying in a three-regular tree component lies in the same component.

\begin{lemma}[Clopen selection on the three-regular tree part]\label{lem:tree-part}
There is a set $C_Z\subseteq Z$, clopen in $Z$, such that
\begin{enumerate}[label=\textup{(\roman*)},leftmargin=2.2em]
\item $C_Z$ meets every three-point $p$-orbit in $Z$ in exactly one point;
\item $C_Z\cap qC_Z=\varnothing$.
\end{enumerate}
Thus the corresponding edges give a clopen collision-free choice at every vertex of $Z_3$.
\end{lemma}

\begin{proof}
Extend $p$ and $q$ by the identity outside $Y$ on the two-fold amplification $\widehat X=X\times[2]$.  Consider the simple clopen graph $L$ on $\widehat X$ in which
\begin{equation}
   y\sim_L z
   \quad\Longleftrightarrow\quad
   y\ne z\ \text{ and }\ z\in\{py,p^2y,qy\}.
   \label{eq:clopen-triangle-graph}
\end{equation}
It has degree at most three.  Each triple $J_1(a),J_2(a),J_3(a)$ spans a triangle, and each $b\in B$ contributes the edge between $(b,1)$ and $(b,2)$.  On $Z$, the graph is obtained from $H|_{Z_3}$ by replacing each vertex $v$ by the triangle $\{v,pv,p^2v\}$ and joining the appropriate triangle vertices by the $q$-edges.

The orbit relation is topologically amenable, so the bounded-degree clopen graph $L$ on the finite amplification is topologically amenable by the Preliminaries. Apply \eqref{eq:localized-kernel} to $L$ with error $1/20$.  We obtain a common radius $R_0$ and continuous probabilities $\lambda_y$ such that
\begin{equation}
    \supp\lambda_y\subseteq \Ball_L(y,R_0)
    \label{eq:clopen-tree-kernel-support}
\end{equation}
and
\begin{equation}
 \|\lambda_y-\lambda_{py}\|_1<\frac1{20},\qquad
 \|\lambda_y-\lambda_{p^2y}\|_1<\frac1{20},\qquad
 \|\lambda_y-\lambda_{qy}\|_1<\frac1{20}.
 \label{eq:clopen-tree-kernel-variation}
\end{equation}
Because $Z$ is $p$- and $q$-invariant, the support of $\lambda_v$ stays inside $Z$ whenever $v\in Z_3$.  Define a probability on the vertices of the tree component of $v$ by
\begin{equation}
       \nu_v=\pi_*\lambda_v.
       \label{eq:clopen-tree-pushforward}
\end{equation}
If $v,w\in Z_3$ are adjacent in $H$, then for some $i,j\in\{0,1,2\}$,
\[
       w=p^{-j}qp^iv.
\]
Applying \eqref{eq:clopen-tree-kernel-variation} successively along the clopen orbit maps $p^i$, $q$, and $p^{-j}$, and using contraction of $\ell^1$-distance under the same pushforward,
\begin{equation}
       \|\nu_v-\nu_w\|_1<\frac3{20}<\frac14.
       \label{eq:clopen-tree-edge-variation}
\end{equation}

Let $e=\{v,w\}$ be an edge of the tree.  After removing $e$, denote by $T_{v\to w}$ the component containing $w$, and set
\begin{equation}
       f(v,e)=\nu_v(T_{v\to w}).
       \label{eq:clopen-sector-weight}
\end{equation}
For each of the three incident edge labels at $v$, the function $f(v,e)$ is continuous on $Z_3$.  Indeed, \eqref{eq:clopen-tree-kernel-support} gives a uniform bound $R$ on the $H$-distance from $v$ to the support of $\nu_v$.  For a support point $u$, the condition $u\in T_{v\to w}$ is decided inside the radius-$(R+1)$ ball around $v$: in a tree, it says exactly that the unique path from $v$ to $u$ starts with $e$.  This is a fixed-radius clopen condition, while the finitely many weights in $\nu_v$ vary continuously.

The three half-trees at $v$ are disjoint, so
\begin{equation}
       \min_{e\ni v} f(v,e)\le \frac13.
       \label{eq:clopen-small-sector}
\end{equation}
At the two ends of an edge the corresponding half-trees are complementary.  Therefore, by \eqref{eq:clopen-tree-edge-variation},
\begin{equation}
       f(v,e)+f(w,e)
       \ge 1-\|\nu_v-\nu_w\|_1
       >\frac34.
       \label{eq:clopen-opposite-sectors}
\end{equation}
Choose
\[
       \frac13<c<\frac38.
\]
At every vertex some incident edge satisfies $f(v,e)<c$ by \eqref{eq:clopen-small-sector}, while by \eqref{eq:clopen-opposite-sectors} no edge satisfies this inequality at both ends.

Label the three incident edges at $v\in Z_3$ by the three points $v,pv,p^2v$.  For $i=0,1,2$, let $U_i$ be the open set on which the edge carrying the label $p^iv$ satisfies $f<c$.  The sets $U_0,U_1,U_2$ cover $Z_3$.  Since $Z_3$ is compact and zero-dimensional, there is a clopen partition
\[
       Z_3=C_0\sqcup C_1\sqcup C_2
\]
with $C_i\subseteq U_i$.  Set
\begin{equation}
       C_Z=C_0\sqcup pC_1\sqcup p^2C_2.
       \label{eq:clopen-tree-transversal}
\end{equation}
This set is clopen in $Z$ and meets each $p$-orbit exactly once.  The chosen edge at a vertex is the one represented by the selected point of its triangle.  Since no tree edge satisfies $f<c$ at both ends, no edge is selected from both ends, which is exactly $C_Z\cap qC_Z=\varnothing$.
\end{proof}

\begin{lemma}[A clopen partial matching extending the tree choice]\label{lem:extend-tree}
There is a clopen partial matching $M_0$ in the six charts \eqref{eq:clopen-six-charts} that covers every $a\in A$ such that $J_1(a)\in Z_3$.
\end{lemma}

\begin{proof}
The set $C_Z$ is clopen in the closed set $Z$, and $C_Z$ and $Z\setminus C_Z$ are disjoint compact sets.  We choose $D$ to separate $C_Z$ from $Z\setminus C_Z$ inside $F$; such a clopen separator exists because $F$ is a Cantor space.  Thus there is a clopen set $D\subseteq F$ with
\[
       D\cap Z=C_Z.
\]
Set
\begin{equation}
       C^{(0)}=D\setminus\bigl(pD\cup p^2D\cup qD\bigr).
       \label{eq:clopen-extended-transversal}
\end{equation}
Then $C^{(0)}$ satisfies \eqref{eq:clopen-partial-matching}: it is contained in $D$ and avoids $pD$ and $p^2D$, so $C^{(0)},pC^{(0)},p^2C^{(0)}$ are pairwise disjoint, and its avoidance of $qD$ gives $C^{(0)}\cap qC^{(0)}=\varnothing$.  Hence it defines a clopen partial matching $M_0$.

Since $Z$ is invariant under $p$ and $q$ and $D\cap Z=C_Z$, we have
\[
   (pD)\cap Z=pC_Z,\qquad
   (p^2D)\cap Z=p^2C_Z,\qquad
   (qD)\cap Z=qC_Z.
\]
Therefore
\[
   C^{(0)}\cap Z
   =C_Z\setminus\bigl(pC_Z\cup p^2C_Z\cup qC_Z\bigr)
   =C_Z,
\]
where the last equality follows from Lemma~\ref{lem:tree-part}.  If $v_a=J_1(a)\in Z_3$, then $C_Z$ meets the three-point $p$-orbit $\{J_1(a),J_2(a),J_3(a)\}$ exactly once.  Hence exactly one of these three points belongs to $C^{(0)}$, so $M_0$ matches $a$.  Thus $M_0$ covers every $a\in A$ such that $J_1(a)\in Z_3$.
\end{proof}

Let
\begin{equation}
       U_0=A\setminus\dom M_0.
       \label{eq:clopen-initial-unmatched}
\end{equation}
This set is clopen, hence compact.  If $U_0=\varnothing$, the proof is complete.  Otherwise the increasing family
\[
       \{J_1^{-1}(W_n):n\ge1\}
\]
is an open cover of $U_0$, because every tree vertex was already covered.  Choose $N$ such that
\begin{equation}
       J_1(U_0)\subseteq W_N,
       \label{eq:clopen-certificate-radius}
\end{equation}
and let $r$ be an upper bound for the number of vertices in a radius-$N$ ball of the degree-at-most-three multigraph $H$.

\begin{lemma}[Finite certificates outside the tree part]\label{lem:certificate}
For every $a\in U_0$ there are a set
\[
       a\in S_a\subseteq A,\qquad |S_a|\le r,
\]
and an injection
\[
       j_a:S_a\longrightarrow B
\]
along the allowed edges \eqref{eq:clopen-six-charts}, such that every left neighbour, in the bipartite multigraph defined by the six charts \eqref{eq:clopen-six-charts}, of every point of $j_a(S_a)$ belongs to $S_a$.

Consequently, for every matching $M$ that leaves $a$ unmatched, there is an $M$-augmenting path starting at $a$ of length at most $2|S_a|-1\le 2r-1$.
\end{lemma}

\begin{proof}
Throughout this proof, all balls, paths, cycles, and selected edges are taken in the multigraph $H$; we return to the bipartite chart multigraph only when translating the resulting finite configuration back to $A\sqcup B$.

Fix $a\in U_0$ and work inside the finite multigraph induced by the radius-$N$ ball around $v_a$.  By \eqref{eq:clopen-certificate-radius}, this ball contains either a vertex of $V_1$ or a cycle.  At this point we only need to construct one certificate for this fixed $a$; no clopen or continuous choice of the objects below is required.
If the ball contains vertices of $V_1$, choose one whose $H$-distance from $v_a$ is minimal, and take a shortest path from $v_a$ to that vertex.  At each degree-three vertex of the path, select the next edge toward the endpoint.  The selected edges are distinct, and every degree-three endpoint of a selected edge lies on the same path.

Otherwise the ball contains a cycle.  Choose a simple cycle whose vertex set has minimal $H$-distance from $v_a$, and take a shortest path, still inside the same ball, from $v_a$ to the vertex set of this cycle.  The path meets the cycle only at its endpoint.  Orient the cycle.  At each cycle vertex select its outgoing cycle edge, and at every other degree-three vertex on the path select the next edge toward the cycle.  Again the selected edges are distinct, and every degree-three endpoint of a selected edge belongs to the chosen path-and-cycle set.  A loop and a two-edge cycle are included in the same description.

Let $R_a$ be the set of degree-three vertices on the chosen path, or on the chosen path together with the cycle, respectively, and set $S_a=\{a'\in A: v_{a'}\in R_a\}$.  Translate each selected labelled edge $e_b$ at $v_{a'}$ to the assignment $j_a(a')=b$.  The chosen vertices all lie in the radius-$N$ ball, hence $|S_a|\le r$.  The endpoint property just verified says precisely that every left neighbour of $j_a(S_a)$ lies in $S_a$.  This completes the construction of the certificate.

Now let $M$ be any matching leaving $a$ unmatched.  Starting at $a$, follow the $j_a$-edge forward.  If the resulting right vertex is matched, follow its unique $M$-edge backward, then repeat.  Every left vertex encountered belongs to $S_a$.  No left vertex can be visited twice.  Indeed, revisiting a left vertex would force the preceding right vertex to be the same, since a left vertex is incident to at most one edge of $M$; injectivity of $j_a$ then forces the preceding left vertex to be the same as well.  Repeating this argument backwards would eventually require an $M$-edge entering $a$, which is impossible because $a$ is unmatched.  Since $S_a$ is finite, the process therefore stops at an unmatched right vertex after at most $|S_a|$ forward steps.

The right vertices are distinct because $j_a$ is injective.  At a later left vertex, the matching edge returns to the preceding right vertex, whereas the new $j_a$-edge has a different right endpoint.  Thus every forward edge is outside $M$ and every backward edge lies in $M$.  The chain is therefore an augmenting path, of length at most $2|S_a|-1$.
\end{proof}

We use the following clopen form of the finite-stage augmentation argument.

\begin{lemma}[Finite clopen augmentation]\label{lem:clopen-augmentation}
Let a bipartite multigraph with left and right vertex spaces $L$ and $R$ be presented by finitely many clopen charts
\[
       \phi_s:D_s\to R_s.
\]
Given a clopen matching $M_0$ and an integer $r\ge1$, finitely many clopen operations produce a clopen matching $M$ with no augmenting path of length at most $2r-1$.  Every vertex matched by $M_0$ remains matched.
\end{lemma}

An augmenting path starts at an unmatched left vertex, ends at an unmatched right vertex, and alternates between edges outside the matching and edges in the matching. Its first and last edges are outside the matching. A matching is clopen when, in each chart, the set of roots of selected edges is clopen.

\begin{proof}
It is enough to explain one stage.  Suppose that the current matching $M$ has no augmenting path shorter than
\[
       \ell=2k+1.
\]
Direct every edge outside $M$ from left to right and every edge in $M$ from right to left.  Let $d(v)$ be the directed distance from the set of unmatched left vertices, with value $\infty$ if $v$ is unreachable.

If
\begin{equation}\label{eq:claug-path}
       (v_0,v_1,\ldots,v_\ell)
\end{equation}
is an augmenting path of length $\ell$, then
\begin{equation}\label{eq:claug-distance}
       d(v_j)=j\qquad(0\le j\le\ell).
\end{equation}
Indeed, a shorter directed path to $v_j$, followed by the remaining tail of \eqref{eq:claug-path}, would give a shorter directed walk from an unmatched left vertex to an unmatched right vertex; after erasing repetitions one obtains a shorter augmenting path.

A length-$\ell$ alternating walk is determined by its initial left vertex and the finite string of chart labels used on its forward edges; the backward steps are forced by $M^{-1}$.  There are only finitely many such strings.  For a fixed string $s$, the set $D_s^{\mathrm{root}}$ of roots producing an alternating walk with unmatched endpoints is clopen: all chart-domain conditions, exclusions of matching edges, and endpoint conditions are clopen.  Since no shorter augmenting path exists, every such walk is already a path.  Its position maps
\[
       t_{s,j}:D_s^{\mathrm{root}}\to L\sqcup R,
       \qquad 0\le j\le\ell,
\]
are homeomorphisms onto clopen sets.  Two paths with the same string and different roots cannot meet.  At the same position this follows from injectivity of the position map; at different positions it would contradict \eqref{eq:claug-distance}.

Process the finitely many strings in a fixed order.  Suppose $W$ is the clopen set of vertices lying on paths already retained.  For the next string retain exactly the roots in
\[
       D_s^{\mathrm{root}}
       \setminus
       \bigcup_{j=0}^{\ell} t_{s,j}^{-1}(W).
\]
Update $W$ by adding the vertices of these retained paths.  These are finite clopen operations.  At the end we have a pairwise vertex-disjoint family that is maximal among the old length-$\ell$ augmenting paths.  Flip the matching simultaneously along all retained paths, obtaining $M'$; no previously matched vertex becomes unmatched.

We claim that $M'$ has no augmenting path of length at most $\ell$.  For an edge whose matching status did not change, its orientation for $M'$ increases the old distance $d$ by at most one.  Every edge whose status changed belongs to a retained path; by \eqref{eq:claug-distance}, its new orientation decreases $d$ by one.

Suppose that $Q$ were an $M'$-augmenting path of length $h\le\ell$.  Both endpoints of $Q$ were also unmatched for $M$, since the flip only matches previously unmatched endpoints and never unmatches a matched vertex.  Its initial old distance is therefore zero, while the terminal old distance is at least $\ell$ by the inductive assumption.  Starting at the initial vertex, the edge inequalities also show that all old distances along $Q$ are finite. Following $Q$ shows that necessarily $h=\ell$ and that $d$ increases by one at every step.  Hence $Q$ uses no changed edge.

Moreover $Q$ cannot meet any retained path.  Every vertex on a retained path is matched in $M'$, and its unique $M'$-matching edge on that path is a changed edge.  An augmenting path passing through such a vertex would have to use that matching edge.  Thus $Q$ is an old length-$\ell$ augmenting path disjoint from the maximal retained family, a contradiction.

Starting with $M_0$, perform this stage successively for
\[
       \ell=1,3,\ldots,2r-1.
\]
There are finitely many stages, and each stage preserves every previously matched vertex.  The final matching has the required property.
\end{proof}

We now finish the proof of the proposition.  By Lemmas~\ref{lem:tree-part} and~\ref{lem:extend-tree}, we have the clopen partial matching $M_0$ covering the whole three-regular tree part.  If $M_0$ already covers $A$, the correspondence in the setup gives the required injection.

Otherwise choose $N$ and $r$ as in \eqref{eq:clopen-certificate-radius}.  Apply Lemma~\ref{lem:clopen-augmentation} to $M_0$ with this value of $r$.  We obtain a clopen matching $M$ with no augmenting path of length at most $2r-1$, and every vertex matched by $M_0$ remains matched.

Suppose that some left vertex $a$ were still unmatched.  Since every left vertex covered by $M_0$ stays covered, we have $a\in U_0$.  Lemma~\ref{lem:certificate} then produces an $M$-augmenting path starting at $a$ of length at most $2r-1$, a contradiction.  Hence $M$ covers all of $A$.

By the six charts \eqref{eq:clopen-six-charts}, this matching gives a clopen injection $I:A\to B$.  By construction, $I(a)$ is always the underlying $B$-coordinate of one of $J(a,1),J(a,2),J(a,3)$.  Every construction above uses only finitely many clopen charts and finite sheet labels, so the same proof works verbatim in finite amplifications.
\end{proof}

\subsection{Dynamical comparison}

We use the clopen type semigroup $\mathsf S$ introduced in the Preliminaries and put $u=[X]$.

\begin{proposition}[Comparison]\label{prop:comparison}
Let $\Gamma\curvearrowright X$ be minimal, and suppose that its orbit relation is topologically amenable. If $A,B\subseteq X$ are clopen, $B\ne\varnothing$, and
\[
 \mu(A)<\mu(B)\qquad(\mu\in\M),
\]
then $A\preceqcl B$.
\end{proposition}

\begin{proof}
Since $\Gamma\curvearrowright X$ is minimal, every nonzero $a\in\mathsf S$ is an order unit: every $b\in\mathsf S$ satisfies $b\leq na$ for some $n$. To see this, a nonzero representative has a nonempty clopen slice $U\subseteq X$. Finitely many $\Gamma$-translates of $U$ cover $X$, so $u\leq n[U]\leq na$. Every clopen set in a finite amplification has type at most $ku$ for some $k$.

An ordered monoid with order unit $u$ has \emph{weak comparability} if
\[
 \forall a\ne0\ \exists k\geq1\ \forall b:
 \qquad kb\leq u\ \Longrightarrow\ b\leq a.
\]
Fix $a\ne0$ and choose $h\geq1$ with $u\leq2^ha$. If $3^hb\leq u$, then
\[
 3^hb\leq u\leq2^ha.
\]
Apply Proposition~\ref{prop:clopen-three-two} in finite amplifications, successively to $3^{h-1}b,2^{h-1}a$ and to the smaller powers. After $h$ applications this gives $b\leq a$. Thus $\mathsf S$ has weak comparability.

Melleray \cite[Proposition~2.20(1)]{Melleray} shows that, for a minimal action on a compact zero-dimensional space, weak comparability is equivalent to almost unperforation of the clopen type semigroup. The latter means that $(n+1)a\leq nb$ implies $a\leq b$. By \cite[Proposition~2.18]{Melleray}, it is equivalent to dynamical comparison, which gives the asserted clopen injection.
\end{proof}

\subsection{A packing large for every invariant measure}

The next lemma turns finitely many candidate packings into one packing which is large simultaneously for all invariant measures. The replacements are made inside finite tiles, and every replacement is nondecreasing for every invariant measure. This is the part of the local replacement method of \cite[Proposition~3.1]{ElekTimarUBA} that we need here.

\begin{lemma}[Measure-monotone replacement]\label{lem:measure-replacement}
Let $\Gamma\curvearrowright X$ be a stable action as in the Preliminaries. Let $\mathcal P_1,\ldots,\mathcal P_m$ be clopen packings, and let $H$ be a bounded-degree clopen graph in which every tile of every $\mathcal P_i$ is a clique. Given full clopen tilings $\mathcal E_1,\ldots,\mathcal E_q$, one can construct a clopen packing $\mathcal T$, using only tiles of the $\mathcal P_i$, such that
\begin{equation}\label{eq:replacement-conclusion}
 \mu([\mathcal T])\geq
 \frac1m\sum_{i=1}^m\mu([\mathcal P_i])
 -\frac1q\sum_{j=1}^q\mu(D_{H^{[2]}}(\mathcal E_j))
 \qquad(\mu\in\M).
\end{equation}
\end{lemma}

\begin{proof}
We call the tiles occurring in $\mathcal P_1,\ldots,\mathcal P_m$ the candidate tiles. Starting with the empty packing $\mathcal T_0$, we recursively define packings $\mathcal T_1,\ldots,\mathcal T_q$. At the $j$-th step, the tiling $\mathcal E_j$ is used to modify $\mathcal T_{j-1}$ and produce $\mathcal T_j$, as follows.

\emph{The replacement.}
Keep all tiles of $\mathcal T_{j-1}$ that are not contained in a single $\mathcal E_j$-tile. Call them crossing tiles, and let $B_j$ be their union. For each $\mathcal E_j$-tile $Q$, consider all candidate tiles contained in $Q\setminus B_j$. Choose a pairwise disjoint family of these candidates for which the number of covered points is maximal. The new packing $\mathcal T_j$ consists of these chosen tiles, over all $Q$, together with the unchanged crossing tiles.

This is a finite clopen operation. Crossing is decided by finitely many incidence tests, so $B_j$ is clopen. In a tile $Q$, each $\mathcal P_i$ contributes at most $|Q|$ candidates contained in $Q$; thus there are uniformly finitely many candidates and subfamilies to test. Refine the tower bases of $\mathcal E_j$ so that containment, disjointness and cardinality tests are constant. On each refined base choose the first maximizing family in a fixed order. The choice is made once on the base and transported to its levels, not independently at different points of the same tile. Choices in different $Q$'s are disjoint and avoid $B_j$. Hence $\mathcal T_j$ is a clopen packing.

\emph{Monotonicity.}
Every old tile contained in $Q$ was disjoint from the crossing tiles and is still an available candidate in $Q\setminus B_j$. Thus the old contained tiles give one admissible choice. The crossing tiles themselves do not change. Consequently, for every tile $Q\in\mathcal E_j$,
\[
 |[\mathcal T_j]\cap Q|\geq |[\mathcal T_{j-1}]\cap Q|.
\]
By \eqref{eq:tower-average}, for every $\mu\in\M$,
\begin{equation}\label{eq:replacement-monotonicity}
 \mu([\mathcal T_j])\geq\mu([\mathcal T_{j-1}]).
\end{equation}

\emph{Comparison with the original packings.}
Put
\[
 W_j=X\setminus D_{H^{[2]}}(\mathcal E_j).
\]
Thus $x\in W_j$ means that $B_H(x,2)\subseteq\mathcal E_j(x)$. Fix $i$, a tile $Q\in\mathcal E_j$, and $x\in[\mathcal P_i]\cap W_j\cap Q$. The $\mathcal P_i$-tile $T$ containing $x$ is a clique in $H$, so
\[
 T\subseteq B_H(x,1)\subseteq Q.
\]
It cannot meet a crossing tile $T'$ of the old packing. Indeed, $T'$ is also a candidate and hence an $H$-clique. If $z\in T\cap T'$, every point of $T'$ is at $H$-distance at most two from $x$, whence
\[
 T'\subseteq B_H(x,2)\subseteq Q,
\]
contrary to its being a crossing tile.

It follows that all $\mathcal P_i$-tiles meeting $W_j\cap Q$ are available candidates in $Q\setminus B_j$. For this fixed $i$ they are pairwise disjoint. Maximality of the covered cardinality therefore gives
\[
 |[\mathcal T_j]\cap Q|
 \geq |[\mathcal P_i]\cap W_j\cap Q|.
\]
Applying \eqref{eq:tower-average} once more yields, for every $i=1,\ldots,m$ and every $\mu\in\M$,
\begin{equation}\label{eq:replacement-comparison}
 \mu([\mathcal T_j])
 \geq\mu([\mathcal P_i])-\mu(D_{H^{[2]}}(\mathcal E_j)).
\end{equation}

Take $\mathcal T=\mathcal T_q$. By \eqref{eq:replacement-monotonicity}, its measure dominates that of every intermediate packing, for every invariant measure. Thus the right-hand side of \eqref{eq:replacement-comparison} is a lower bound for $\mu([\mathcal T])$ for every pair $i,j$. Averaging these lower bounds proves \eqref{eq:replacement-conclusion}. All tiles of $\mathcal T$ remain among the original candidates, so their common finite orbit bound is unchanged.
\end{proof}

\begin{proposition}[A large F\o lner packing]\label{prop:large-packing}
Let $\Gamma\curvearrowright X$ be a stable action as in the Preliminaries, and suppose that its orbit relation is topologically amenable. For every bounded-degree clopen graph $K$ and every $0<t,a<1$, there is a clopen packing $\mathcal T$ such that every tile is $t$-F\o lner for $K$ (as in the Preliminaries), and
\begin{equation}\label{eq:large-packing}
 \mu([\mathcal T])>1-a\qquad(\mu\in\M).
\end{equation}
The packing has a finite orbit bound in the sense of \eqref{eq:orbit-bounded-tiles} of the Preliminaries.
\end{proposition}

\begin{proof}
The graph $K$ is topologically amenable by the Preliminaries. Apply Proposition~\ref{prop:random-tilings} with $H=K$ and $\beta=ta/2$. We obtain full clopen tilings $\mathcal R_1,\ldots,\mathcal R_m$ such that, for every $x\in X$,
\[
 \frac1m\sum_{i=1}^m\1_{D_K(\mathcal R_i)}(x)<\frac{ta}{2}.
\]
From $\mathcal R_i$ retain precisely the $t$-F\o lner tiles; denote the resulting packing by $\mathcal P_i$ and its union by $Y_i$. These are clopen because the tile cardinalities and their $K$-boundary cardinalities are clopen finite data.

On a discarded tile the ratio
\[
 r_i(x)=\frac{|\partial_K\mathcal R_i(x)|}{|\mathcal R_i(x)|}
\]
is at least $t$. By \eqref{eq:boundary-average},
\[
 \mu(X\setminus Y_i)\leq\frac1t\int_Xr_i\,d\mu
 =\frac1t\mu(D_K(\mathcal R_i)).
\]
Averaging and using the pointwise estimate from Proposition~\ref{prop:random-tilings}, we obtain that, for every $\mu\in\M$,
\begin{equation}\label{eq:average-good-packing}
 \frac1m\sum_{i=1}^m\mu(Y_i)>1-\frac a2.
\end{equation}
This is a measure estimate; it does not assert large pointwise coverage by the $\mathcal P_i$.

Choose a bounded-degree clopen graph $H$ making every tile of every $\mathcal P_i$ a clique. Such a graph exists by their finite tower descriptions. The graph $H^{[2]}$ is topologically amenable by the Preliminaries. Proposition~\ref{prop:random-tilings} therefore gives full clopen tilings $\mathcal E_1,\ldots,\mathcal E_q$ such that, for every $x\in X$,
\[
 \frac1q\sum_{j=1}^q
 \1_{D_{H^{[2]}}(\mathcal E_j)}(x)<\frac a2.
\]
Integrating, we obtain, for every $\mu\in\M$,
\[
 \frac1q\sum_{j=1}^q
 \mu(D_{H^{[2]}}(\mathcal E_j))<\frac a2.
\]
Lemma~\ref{lem:measure-replacement} and \eqref{eq:average-good-packing} now give a single packing $\mathcal T$ with
\[
 \mu([\mathcal T])
 \geq \frac1m\sum_i\mu(Y_i)
       -\frac1q\sum_j\mu(D_{H^{[2]}}(\mathcal E_j))
 >1-a
\]
for every $\mu\in\M$. Every tile of $\mathcal T$ is a tile of some $\mathcal P_i$, so it is still $t$-F\o lner for $K$. Neither the choices nor the resulting packing depend on $\mu$.
\end{proof}

\subsection{Absorption and strong almost finiteness}

First, we prove almost finiteness.

\begin{proposition}[Almost finiteness]\label{prop:clopen-almost-finite}
Suppose that $\Gamma\curvearrowright X$ is minimal, that its orbit relation is topologically amenable, and that $\M\ne\varnothing$. Then the orbit relation is almost finite.
\end{proposition}

\begin{proof}
Fix a bounded-degree clopen graph $K$ and $\tau>0$, and choose
\[
 0<t<\min\{1/10,\tau/5\}.
\]
Every full orbit is infinite by minimality. Hence $K$ can be enlarged to a bounded-degree clopen graph $K_0$ such that
\begin{equation}\label{eq:large-neighbourhood}
 |K_0x|>1/t\qquad(x\in X).
\end{equation}
For each point, choose more than $1/t$ distinct orbit points different from it. These inequalities persist on a clopen neighbourhood. A finite subcover of $X$ and the finitely many resulting orbit maps give $K_0$. This uses the full orbit structure, not minimality of $K$.

Apply Proposition~\ref{prop:large-packing} to $K_0$ with $a=t$. Let $\mathcal T$ be its packing and $Y=[\mathcal T]$. Every tile $T$ has a $K_0$-interior point, since $|\partial_{K_0}T|<t|T|$ and $t<1$. By \eqref{eq:large-neighbourhood},
\[
 |T|>1/t.
\]
In each tile select a reserve $B_T$ of size $\lceil3t|T|\rceil$. The tower levels give this selection clopen, and
\begin{equation}\label{eq:reserve-size}
 3t|T|\leq |B_T|<4t|T|<|T|.
\end{equation}
Put
\[
 A=X\setminus Y,\qquad B=\bigcup_{T\in\mathcal T}B_T.
\]
The sets $A,B$ are clopen. By invariance, tower counting and \eqref{eq:reserve-size},
\begin{equation}\label{eq:reserve-measure}
 \mu(A)<t,
 \qquad \mu(B)\geq3t\mu(Y)>3t(1-t)>t
 \qquad(\mu\in\M).
\end{equation}
Since $\M\ne\varnothing$, the set $B$ is nonempty. If $A=\varnothing$, the packing is already the required tiling. Otherwise Proposition~\ref{prop:comparison} gives a clopen orbit injection $\psi:A\to B$.

Replace each tile by
\begin{equation}\label{eq:absorbed-tile}
 T'=T\cup\psi^{-1}(B_T).
\end{equation}
These sets are disjoint and cover $X$. They have a finite clopen tower description. Indeed, refine each old tower base so that the reserve levels, membership in $\ran\psi$, and the inverse branch of $\psi$ at each relevant level are fixed. Add the corresponding inverse images as new levels. Their ranges lie in $A$, hence are disjoint from the old levels; they are mutually disjoint because $\psi$ is injective. All level maps are finite compositions of clopen orbit maps. This also supplies a common finite orbit bound for the new tiles.

Finally,
\[
 |T'\setminus T|\leq |B_T|<4t|T|,
 \qquad \partial_KT'\subseteq\partial_{K_0}T\cup(T'\setminus T).
\]
Consequently,
\begin{equation}\label{eq:absorbed-boundary}
 |\partial_KT'|<5t|T|\leq5t|T'|<\tau|T'|.
\end{equation}
Hence the orbit relation is almost finite.
\end{proof}

\medskip
\begin{proof}[Proof of Theorem~\ref{thm:clopen-main}]
Now we prove strong almost finiteness. Fix $K$ and $\tau,\eta>0$. By Proposition~\ref{prop:clopen-almost-finite}, choose a clopen tiling $\mathcal T$ all of whose tiles are $\tau$-F\o lner for $K$. Let $H$ contain $K$ and make each $\mathcal T$-tile a clique. Let $d$ bound the degree of $K$, and choose $\theta>0$ with $(d+1)\theta<\eta$.

The graph $H$ is topologically amenable by the Preliminaries. Corollary~\ref{cor:separated-packings}, with separation distance one in $H$, gives clopen packings $\mathcal P_1,\ldots,\mathcal P_m$ satisfying \eqref{eq:separated-coverage}. We use only their separation and pointwise coverage; their tiles need not be F\o lner.

For each $P\in\mathcal P_i$, set
\[
 \widehat P=\bigcup\{T\in\mathcal T:T\cap P\ne\varnothing\}.
\]
These sets are pairwise disjoint. Otherwise one $\mathcal T$-tile would meet two different $\mathcal P_i$-tiles, and two points of those intersections would be joined by an $H$-edge, contrary to separation. Add all $\mathcal T$-tiles meeting no tile of $\mathcal P_i$. The resulting family $\widehat{\mathcal P}_i$ is a full clopen tiling. Its tiles are uniformly bounded unions of old tiles, and all incidence tests are clopen. More explicitly, two points in an enlarged tile are joined by an old-tile map, a $\mathcal P_i$-tile map and another old-tile map. Finitely many such compositions describe the new finite equivalence relation and give its tower form.

Every new tile is a disjoint union of $\tau$-F\o lner $\mathcal T$-tiles. Since
\[
 \partial_K\widehat P\subseteq
 \bigcup_{T\subseteq\widehat P}\partial_KT,
\]
it too is $\tau$-F\o lner for $K$.

Fix $x\in X$. For each of the at most $d+1$ points in $\{x\}\cup Kx$, fewer than $\theta m$ packings fail to cover that point. Therefore, in all but $(d+1)\theta m<\eta m$ indices, $x$ and all its $K$-neighbours are covered by $\mathcal P_i$. They belong to one $\mathcal P_i$-tile: $K\subseteq H$, and an $H$-edge cannot join different tiles of this packing. Thus $x$ is already interior to that tile, and remains interior after enlargement. This is precisely \eqref{eq:strong-conclusion} for the tilings $\widehat{\mathcal P}_1,\ldots,\widehat{\mathcal P}_m$. Hence strong almost finiteness follows.
\end{proof}

\subsection{The equivalent F\o lner formulation}

The \emph{orbit F\o lner property} means that, for every bounded-degree clopen graph $K$ and $\varepsilon>0$, there is a nonempty finite set $F$ in one full orbit with $|\partial_KF|<\varepsilon|F|$. This is an existence condition, not a uniform bound in the metric of each $K$. It is the F\o lner property of the orbit structure used in the preliminaries.

\begin{corollary}\label{cor:clopen-equivalence}
For a minimal stable orbit structure on a Cantor space, the following are equivalent:
\begin{enumerate}[label=\textup{(\roman*)}]
\item it is topologically amenable and admits an invariant probability measure;
\item it is topologically amenable and has the orbit F\o lner property;
\item it is strongly almost finite.
\end{enumerate}
\end{corollary}

\begin{proof}
Theorem~\ref{thm:clopen-main} proves (i)$\Rightarrow$(iii).

Assume (iii). A tile of a sufficiently good tiling gives the orbit F\o lner property. To prove amenability, fix $K$ and $\varepsilon>0$, and choose tilings $\mathcal E_1,\ldots,\mathcal E_m$ for which the boundary frequency is less than $\varepsilon/2$. Define
\[
 p_x=\frac1m\sum_{i=1}^m
 \frac1{|\mathcal E_i(x)|}\sum_{z\in\mathcal E_i(x)}\delta_z.
\]
The tower descriptions give locally constant rational weights with a common denominator and a common finite orbit bound on the support. For a $K$-edge $xy$, the two summands agree whenever $x,y$ belong to the same $\mathcal E_i$-tile; otherwise their $\ell^1$-distance is at most two. The latter event implies $x\in D_K(\mathcal E_i)$, so
\[
 \|p_x-p_y\|_1
 \leq\frac2m\sum_i\1_{D_K(\mathcal E_i)}(x)<\varepsilon.
\]
This is the orbitwise topological amenability condition from the Preliminaries and proves (iii)$\Rightarrow$(ii).

Finally, assume (ii). Exhaust $\Gamma$ by increasing finite symmetric sets $S_n$ and let $K_n$ join distinct $x,sx$ for $s\in S_n$. Choose a nonempty finite $F_n$ with
\[
 |\partial_{K_n}F_n|<\frac1n|F_n|,
 \qquad \mu_n=\frac1{|F_n|}\sum_{x\in F_n}\delta_x.
\]
For any fixed $g\in\Gamma$, eventually $g\in S_n$, and then
\[
 |gF_n\mathbin\triangle F_n|\leq2|\partial_{K_n}F_n|.
\]
Hence, for every $f\in C(X)$,
\[
 \left|\int f\circ g\,d\mu_n-\int f\,d\mu_n\right|
 \leq\frac{2\|f\|_\infty}{n}.
\]
Compactness of the space of probabilities on $X$ gives a weak-star convergent subsequence. Its limit is invariant under every $g\in\Gamma$, proving (ii)$\Rightarrow$(i).
\end{proof}

\section{The measure-free case: pure infiniteness}\label{sec:pure}

Corollary~\ref{cor:clopen-equivalence} shows, in particular, that failure of the F\o lner property forces the absence of an invariant probability measure.  We now prove the stronger statement needed for the dichotomy: in the minimal topologically amenable clopen setting, the absence of an invariant probability measure forces pure infiniteness.

Let $S$ be the clopen type semigroup introduced in Section~\ref{sec:prelim}, and put $u=[X]$.  Minimality implies that every nonzero element of $S$ is an order unit.  Indeed, if $A\subseteq X$ is nonempty and clopen, density of every orbit gives finitely many clopen orbit translates of $A$ whose ranges cover $X$; after disjointifying the ranges one obtains
\[
   u\le n[A]
\]
for some $n$. Every nonzero type has a representative with a nonempty clopen slice, so the same order-unit conclusion holds for types in finite amplifications.

Recall that an ordered monoid with order unit $u$ has \emph{weak comparability} if for every $0\neq a$ there is $k\ge1$ such that $kb\le u$ implies $b\le a$ for every $b$.  The three-to-two principle gives this immediately.  Fix $0\neq a$ and choose $h$ with $u\le2^h a$.  If $3^h b\le u$, then
\[
   3^h b\le u\le2^h a.
\]
Applying Proposition~\ref{prop:clopen-three-two} successively to $3^{h-1}b,2^{h-1}a$ and then to the smaller amplifications yields $b\le a$.  Thus $S$ has weak comparability.

Melleray proves that, for a minimal action on a compact zero-dimensional space, weak comparability of the clopen type semigroup implies almost unperforation \cite[Proposition~2.20(1)]{Melleray}.  The fixed stable action from Section~\ref{sec:prelim} has precisely this type semigroup, so the result applies to $S$.  Ara--B\"onicke--Bosa--Li prove that, in the present minimal second-countable setting, almost unperforation of the type semigroup is equivalent to dynamical comparison \cite[Theorem~A]{ABBLType}.  Here almost unperforation means $(n+1)a\leq nb\Rightarrow a\leq b$ for all $a,b$ and $n\geq1$. Hence, for nonempty clopen $A,B\subseteq X$,
\[
   \mu(A)<\mu(B)\quad\text{for every invariant probability measure }\mu
   \qquad\Longrightarrow\qquad
   A\preceqcl B.
\]

Assume now that there is no invariant probability measure.  The hypothesis in the preceding comparison statement is then vacuous, so
\[
   A\preceqcl B
   \qquad\text{for all nonempty clopen }A,B\subseteq X.
\]
If $A\subseteq X$ is nonempty and clopen, write $A=A_1\sqcup A_2$ with $A_1,A_2$ nonempty and clopen.  Then
\[
   A\preceqcl A_1,
   \qquad
   A\preceqcl A_2.
\]
The two targets are disjoint, and therefore $2[A]\le[A]$. Such a decomposition is incompatible with an invariant probability measure when $A=X$.  Every nonempty clopen set is thus $(2,1)$-paradoxical, which is pure infiniteness.

Together with Corollary~\ref{cor:clopen-equivalence} and Theorem~\ref{thm:clopen-main}, this proves the dichotomy stated in Theorem~\ref{thm:main-intro} through the dictionary of Section~\ref{sec:prelim}.
\vskip 0.3in
\noindent

\textbf{Declaration of AI use.}
\vskip 0.1in
\noindent
ChatGPT-6 was used to assist with proof development, proof checking, and manuscript preparation.
The authors reviewed and revised the AI-assisted material and verified the mathematical arguments.
The authors take full responsibility for the content of this work.
\vi
\textbf{Acknowledgement:} The second author was supported by Hungarian National Research, Development and Innovation Office, NKFIH Highlight grant 152849 and Icelandic Research Fund grant 239736-051.

\vi
\emph{G\'abor Elek},  Alfr\'ed R\'enyi Institute of Mathematics, Budapest, Hungary.

\noindent
\texttt{elek.gabor.math[at]gmail.com}
\vi
\emph{\'Ad\'am Tim\'ar}, Division of Mathematics, The Science Institute, University of Iceland, Reykjavik, Iceland
and
Alfr\'ed R\'enyi Institute of Mathematics, Budapest, Hungary.

\noindent
\texttt{madaramit[at]gmail.com}\\ 
\end{document}